\documentclass[11pt,a4paper]{article}

\usepackage[english]{babel}
\usepackage{amssymb,amsmath,amsthm}
\usepackage{amsfonts}
\usepackage{latexsym}
\usepackage{rotating}
\usepackage[]{color}
\usepackage{multirow}
\usepackage{graphicx}
\usepackage{array}
\usepackage{url}

\font\Bbb=msbm10 at 11pt

\newcommand{\RR}{\mbox{\Bbb R}}

\newcommand{\NP}{{\cal NP}}

\newtheorem{corollary}{Corollary}

\newtheorem{lemma}{Lemma}
\newtheorem{proposition}{Proposition}

\font\Bbb=msbm10 at 11pt

\def\RR{\mbox{\Bbb R}}

\def\NP{{\cal NP}}

\newcommand{\xr}[1]{(x^{#1},r^{#1})}
\newcommand{\bla}[1]{{\tilde \lambda}}
\newcommand{\bpi}[1]{{\tilde \pi}}
\usepackage[ruled,linesnumbered]{algorithm2e}

\usepackage{amsmath}
\usepackage{tikz}
\usepackage{pgfplots}

\usepackage{longtable}
\usepackage{booktabs}

\title{A Spatial Benders Algorithm for \\Delay Constrained Routing}

\author{Antonio Frangioni
        \thanks{Dipartimento di Informatica, Universit{\`a} di Pisa,
                Largo B. Pontecorvo 3, 56127 Pisa, Italy.
                E-mail: {\tt antonio.frangioni@unipi.it,~luca.mencarelli@unipi.it}}
\and    Laura Galli
        \thanks{Dipartimento di Matematica,
                Alma Mater Studiorum Università di Bologna,
                Piazza di Porta San Donato 5, 40126 Bologna, Italy.
                E-mail: {\tt l.galli@unibo.it}}
\and    Luca Mencarelli \footnotemark[1]
\and    Enrico Sorbera
        \thanks{Dipartimento di Matematica,
                Universit{\`a} di Trento,
                Via Sommarive 14, 38123 Povo, Trento, Italy.
                E-mail: {\tt enrico.sorbera@unitn.it}}
        }

\date{Draft, 26th September 2026}

\begin{document}

\maketitle

\begin{abstract}
We study the Delay Constrained Routing (DCR) problem arising in IP computer networks to support high bandwidth applications. The goal is to route IP packets, from source to destination, subject to quality of service constraints---in particular, bounding the {\em worst case delay}---while minimizing the total allocated bandwidth in the network. DCR can be formulated as a Mixed-Integer Second-Order Cone Program, and solved using standard solvers; while generally efficient, this hinges on off-the-shelf tools and does not scale to large instances. We devise a solver-free bespoke approach that uses {\em nested Benders-Lagrange}, namely, a double decomposition approach that exploits the special structure of the problem. In particular, we use a nonconvex nonlinear formulation of the problem, that results in a nonconvex Benders value function whose optimization we tackle with a bespoke {\em Spatial Branch-and-Bound} approach using custom piecewise-convex cuts. We show that our approach provides very tight upper and lower bounds in competitive running times w.r.t.~previous state-of-the-art approaches. \\*[2mm]

{\bf Keywords:} Mixed-integer non-linear programming, Benders decomposition, Lagrangian relaxation, Delay-constrained routing.
\end{abstract}

\section{Introduction}\label{sec:intro}

Computer networks provide a wide array of applications that are rich sources of challenging optimization problems. This paper is concerned with the problem of routing packets, from source to destination, in an IP network, subject to {\em Quality of Service} (QoS) constraints. In particular, real-life applications such as voice/video streaming require the {\em worst-case delay} (WCD) of any IP packet, often referred to as ``controlled end-to-end'' delay, to stay below a given threshold. The WCD of a packet depends on both the path along which it is routed, and \emph{in a nonlinear function} from the {\em reserved rate}, i.e., the portion of bandwidth allocated for that flow on the arcs of the path.

\smallskip
\noindent
We focus on the problem of simultaneously computing a path in the IP network (i.e., routing) and reserving rates along it, subject to a WCD delay constraint, in order to minimize the total amount of bandwidth (resource) used. This \emph{Delay Constrained Routing problem} (DCR) was shown in \cite{FGS2015a} to have a natural formulation as a Mixed-Integer Second-Order Cone Program (MISOCP), exploiting the \emph{Perspective Reformulation} (PR) technique \cite{FG2006}. Thus, DCR can be solved using general-purpose Mixed-Integer Non-Linear Programming (MINLP) software for many real/realistic instances, consistently improving the general performance of the telecommunication network \cite{FGS2015b}. Yet, the downside of using MINLP solvers is twofold. First, running times substantially increase with the size of the network. Second, practical implementations in a real-world setting may have issues related to software environment compatibility and licensing, when incorporating and distributing such components.

\smallskip
\noindent
We address these limitations by developing a solver-free, bespoke solution approach that exploits the specific structure of the DCR problem. Our specialised algorithm uses a novel combination of Lagrangian and Benders' decomposition, together with an effective heuristic, which often improves the running time by one order of magnitude with respect to MINLP solvers. Nesting two decomposition techniques is already an uncommon approach, in particular since the problem is nonlinear and therefore the \emph{generalized} Benders' decomposition (GBD) \cite{G1972} need be used. However, the standard GBD approach requires convexity of the formulation (and other properties). We propose a further methodological novelty by rather applying it to a \emph{nonconvex nonlinear} formulation of the problem, that gives rise to a nonconvex value function. Yet, our analysis for the problem allows us to derive a piecewise-convex lower approximation of the value function in each computed point, which leads to implementing a bespoke \emph{Spatial Branch-and-Bound} approach the likes of which does not seem to have ever been previously proposed.

\smallskip
\noindent
The paper is structured as follows. The problem is formally described in Section \ref{sec:prob}, and the relevant literature is reviewed.
In Section \ref{sec:formul} we present two MINLP formulations for the DCR problem.
In Sections \ref{sec:benders}, \ref{sec:lagsolve}, \ref{sec:approx} we present the main ingredients of our algorithm, which is described in detail in Section \ref{sec:algo}.
The results of some extensive computational experiments are given in Section \ref{sec:results}.
Finally, some concluding remarks are made in Section \ref{sec:thend}.

%
%


\section{Problem definition} \label{sec:prob}

We now formally describe the DCR problem. We model an IP computer network by a directed graph $G=(N,A)$, where $N$ is the set of {\em nodes}, representing network elements such as routers or switches, and $A$ is the set of arcs {\em arcs}, representing physical links between the nodes. Each arc $(i,j)\in A$ is characterized by:
\begin{itemize}
 \item {\it physical link speed} $w_{ij}$ (bps), the link bandwidth;
 \item {\it reservable capacity} $c_{ij}$ (bps), the portion of bandwidth available for allocation;
 \item {\it link reservation cost} $f_{ij}$, the cost of reserving one unit of capacity on the link;
 \item {\it link delay} $l_{ij}$, a (fixed) delay due to transmission and physical propagation effects over the link.
\end{itemize}
Each node $i\in N$ is characterized by a (fixed) {\it node processing delay} $n_i$, that a packet incurs when reaches a network node and it is ``processed'' before being forwarded to a queue of packets associated to the node itself.

\smallskip
\noindent
Any application running on a packed-switched network is characterized by a  {\em stream of packets} (or datagrams) originated by a source node and directed to a destination node.
The stream of packets corresponding to an application is often referred to as a {\em flow} in the computer network engineering literature.
This is consistent with the term used in the optimization literature concerning multi-commodity flow (MCF) problems, a classical paradigm for network flow models with multiple {\em commodities} (i.e., flow demands).
In a real-world context, there are usually multiple applications running at the same time, hence multiple packet flows must be simultaneously routed through the network.

\smallskip
\noindent
Although the methods developed in this paper can be applied to several DCR problems, for the sake of brevity and clarity we restrict attention to the  specific {\em Single-Flow Single-Path} (SFSP) case where a single ``new" flow must be routed on a single {\em unsplittable} path from a source $s\in N$ to a destination $d \in N \setminus \{ s \}$ \cite{FGS2015a}, assuming that the other (existing) flows have already been routed and the corresponding resource allocated, thereby reducing the capacity available on the arcs ($c_{ij}$ $\leq w_{ij}$). The flow has a \emph{deadline} $\delta$, representing a threshold, which bounds from above the maximum time that every packet is allowed to spend to be dispatched across the network. Fixed node and link delays are additive, so if we ignore the dynamic component of the problem, (SFSP) DCR reduces to the Constrained Shortest Path problem, which is already $\NP$-Hard.

\smallskip
\noindent
Furthermore, in a packed-switched network all datagrams are subject to a third type of delay, called {\em queuing delay} or {\em latency}, representing the time a packet spends in the node queue before being sent to the link $(i,j)$. The latency depends on the scheduling protocol implemented by the router to manage the queue, and on the portion of bandwidth allocated to the flow, therefore it is strictly related to the dynamic aspect of the process and cannot be represented by a constant. Indeed, the main difference between DCR and classical network optimization problems relies on the fact that, in order to model the latency of a stream of packets, one needs to capture the flow dynamics, which is clearly time-dependent. In particular, the traffic profile of the datagrams can be described as a cumulative function $A(\tau):\mathbb{R}_+\to\mathbb{R}_+$ depending on time $\tau$, called {\em arrival curve}, where $A(\tau)$ is the amount of data (i.e., number of bits) sent by the flow in the interval $[0,\tau)$.
Traffic shaping is a bandwidth management technique used in computer networks to guarantee compliance of datagrams with a desired arrival curve.
As is customary, we assume the classical {\it leaky-bucket} traffic shaper, resulting in an affine function $A(\tau)=\sigma + \rho \, \tau$ with parameters $\sigma$ and $\rho$, called, respectively, {\it burst} and {\it arrival rate}.
To have a finite latency, the reserved rate $r_{ij}$ (i.e., allocated capacity) on each arc $(i,j)$ of the path $P$ must be at least as large as the arrival rate of the flow, i.e.,
\begin{equation}\label{eq:finitedelay}
 r_{ij} \geq \rho ~~\mbox{for all}~~ (i,j) \in P.
\end{equation}
Under \eqref{eq:finitedelay}, worst-case bounds on the actual latency value can be computed using a sophisticated mathematical framework called \emph{network calculus} \cite{LMMS06} for schedulers that work according to the {\em Generalized Processor Sharing} (GPS) \cite{PG1993} paradigm. In particular, for the {\em Strictly Rate-Proportional} (SRP) scheduling protocol \cite{FGS2015a,FGS2015b}, the latency expression for an arc $(i,j)$ is given by
\begin{equation}
 \frac{L}{r_{ij}} + \frac{L}{w_{ij}},
 \label{eq:thetaSRP}
\end{equation}
where the constant $L$ is the \emph{maximum transmit unit} (i.e., the maximum size of any packet). Other GPS-like scheduling protocols give rise to similar formul{\ae}, yet the SRP protocol is the most common practical implementation for GPS. Also, an advantage of the SRP protocol is that the latency of a flow only depends on its path and on the reserved rates along the path: hence, the new flow does not influence the other existing flows, and the only interaction between flows consists in consuming the available capacity on the arcs. Other scheduling protocols, instead, require the introduction of \emph{admission control} constraints due to the impact that each flow has on the feasibility of the other flows \cite{FGS2017a, FGS2017b}. Furthermore, the latency function \eqref{eq:thetaSRP} is convex with respect to the reserved rate $r_{ij} \geq 0$. In fact, the $L \,/\, w_{ij}$ component in \eqref{eq:thetaSRP} is an additive constant term like the fixed delays $n_i$ and $l_{ij}$; therefore, to ease notation, for each arc $(i,j)$ we only distinguish the constant delay ${\bar l}_{ij} = L \,/\, w_{ij} + l_{ij} + n_i$, and the latency  $\theta_{ij} = L \,/\, r_{ij}$ (i.e,  the {\em variable} delay that depends on the {\em reserved rate} $r_{ij}$).

\smallskip
\noindent
With the above assumptions, the WCD of a flow for a path $P$ is given by the closed formula
\begin{equation}
 \frac{\sigma}{r_{\mathrm{min}}} +
 \sum_{(i,j) \in P} \big( \, \theta_{ij} + {\bar l}_{ij} \, \big)
 \label{eq:gendelay}
\end{equation}
where $r_{\mathrm{min}} = \min \{ \, r_{ij} \,:\, (i,j) \in P \}$ captures the minimum reserved rate along the path. A (SFSP) DCR instance is then given by: ($i$) a digraph $G=(N,A)$; ($ii$) a constant delay ${\bar l}_{ij}$, a bandwidth $w_{ij}$, a capacity $c_{ij}$, and a cost $f_{ij}$, for each arc $(i,j) \in A$; ($iii$) a burst parameter $\sigma$ and a rate parameter $\rho$ for the leaky-bucket arrival curve of the flow; ($iv$) a deadline $\delta > 0$. The task is to find an $s$-$d$ path and to reserve rate along its arcs, in order to minimize the total reservation cost, subject to a constraint stating that the WCD, defined by \eqref{eq:gendelay}, must not exceed $\delta$.
%

\section{Formulations} \label{sec:formul}

According to \eqref{eq:gendelay}, the key elements of the WCD of a flow are
\begin{itemize}
  \item the chosen $s$-$d$ {\em path} $P$ in $G$ (a discrete decision);
  \item the \emph{reserved rate}  $0 \leq r_{ij} \leq c_{ij}$ for the arcs in $P$ (continuous decisions).
  %
\end{itemize}
This implies that an algebraic formulation of the problem must contain two types of variables for each $(i,j) \in A$: the former decision can be described using binary variables $x_{ij} \in \{0,1\}$, while the latter corresponds to continuous reserved rate variables $r_{ij}$. Since \eqref{eq:gendelay} contains nonlinear terms with respect to the $r_{ij}$, the (SFSP) DCR problem can be modeled as a MINLP. As typical, the details of the formulation do matter, though.


\subsection{Initial MINLP / MISOCP formulation}\label{ssec:misocp}

In the (SFSP) DCR formulation from \cite{FGS2015a}, the above $x_{ij}$ and $r_{ij}$ are complemented by an auxiliary variable $z$ that captures the minimum reserved rate along the path, i.e., $r_{\mathrm{min}}$ in \eqref{eq:gendelay}; this yields
\begin{align}
 \min \, &  \textstyle \sum_{(i,j)\in A} {f_{ij}r_{ij}}
 & \label{eq:of}\\
 &\textstyle
  \sum_{(j,i)\in BS(i)}{x_{ji}} - \sum_{(i,j)\in FS(i)}{x_{ij}} = b_i
  & i \in N \label{eq:cflow} \\
 &
 \frac{\sigma}{z} +
 \sum_{(i,j) \in A} \Big( \frac{L{x_{ij}}^2}{r_{ij}} + {\bar l}_{ij} x_{ij} \Big)
 \le \delta  & \label{eq:cdelay} \\
 %
 %
 %
 & x_{ij} \in \{0 \,, 1\} &(i,j)\in A \label{eq:cvarx}\\
 & \rho x_{ij} \leq r_{ij} \leq c_{ij} x_{ij}
 & (i,j)\in A \label{eq:sfspbound1} \\
 & \rho \leq z \leq r_{ij} + c_{max} ( 1 - x_{ij} )
 & (i,j)\in A \label{eq:sfsprmin1}
 \end{align}

\noindent
where $b_i$ is $-1$ if $i = s$, $1$ if $i = d$, and $0$ otherwise. The objective function \eqref{eq:of} represents the total reservation cost. The standard flow conservation constraints \eqref{eq:cflow} describe a simple $s$-$d$ path in $G$. The constraint \eqref{eq:cdelay} imposes the WCD constraint \eqref{eq:gendelay}.
Note that, for all arcs $(i,j) \in A$, one needs to impose the following logical conditions on the reserved rates
\begin{equation}\label{ratescond}
\begin{cases}
 r_{ij} = 0 & \text{if } x_{ij}=0 \\
 \rho \leq z \leq r_{ij} \ \ &\text{otherwise}
\end{cases}
\end{equation}
and on the latencies
\begin{equation}\label{eq:latency}
 \theta_{ij}(r_{ij},x_{ij}) = \begin{cases}
                0 \ \ &\text{if } x_{ij}=0\\
                L \,/\, r_{ij} & \text{otherwise}
                \end{cases}
\end{equation}
That is, both $r_{ij}$ and $\theta_{ij}$ are \emph{semi-continuous} variables governed by the same binary variable $x_{ij}$. This is obtained by the {\em Big-M} constraints \eqref{eq:sfspbound1}--\eqref{eq:sfsprmin1}, with the proper ``large'' constant $c_{max} = \max_{(i,j) \in A}\{c_{ij}\}$, combined with the {\em Perspective Reformulation} (PR) technique \cite{FG2006} that yields the (convex) term $Lx_{ij}^2 \,/\, r_{ij}$ in \eqref{eq:cdelay}, which is the \emph{convex envelope} of the (nonconvex) latency term \eqref{eq:latency}. Thus, both nonlinear terms $Lx_{ij}^2 \,/\, r_{ij}$ and $\sigma \,/\, z$ in \eqref{eq:cdelay} are convex; in particular, they can be expressed (introducing auxiliary epigraphical variables) {as conic constraints} \cite{FGS2015a}, which yields a MISOCP formulation of the problem that can be tackled by general-purpose solvers like {\tt Cplex} and {\tt Gurobi}.
%

\subsection{Modified / Nonconvex MINLP formulation}\label{ssec:minlp}

An alternative formulation can be obtained by substituting \eqref{eq:sfspbound1}--\eqref{eq:sfsprmin1} with
\begin{align}
 & z\, x_{ij} \le r_{ij} \le c_{ij}\,x_{ij} & (i,j)\in A  \label{eq:crmin} \\
 & \rho \leq z \leq c_{max}  \label{eq:cvarr}
\end{align}
The two models differ in the way they capture the minimum allocated rate by linking the auxiliary variable $z$ with the reserved rate variables $r_{ij}$. This modification turned out to be crucial for the specific algorithmic technique that we developed. Note that the constraints \eqref{eq:sfspbound1}--\eqref{eq:sfsprmin1} are linear, while the constraints \eqref{eq:crmin} are nonlinear and nonconvex, hence the corresponding MINLP is nonconvex. In order to see the advantage of using the latter, which may seem a complication, it is useful to think of the auxiliary variable $z$ as a parameter, i.e., the variable $z$ is fixed to a certain value in $[ \, \rho \,,\, c_{max} \, ]$.

\begin{proposition}\label{prop:strong}
Assume $z$ is fixed in $[ \, \rho \,,\, c_{max} \, ]$, let ``NLP1(${z}$)''  be the continuous relaxation of the MINLP \eqref{eq:of}--\eqref{eq:sfsprmin1} and ``NLP2(${z}$)'' that of the NLP formed by replacing \eqref{eq:sfspbound1}--\eqref{eq:sfsprmin1} with \eqref{eq:crmin}: then, the lower bound from NLP2($z$) is at least as strong as the lower bound from NLP1($z$), and it can be strictly stronger.
\end{proposition}
\begin{proof}
It suffices to show that any feasible pair $(x^*_{ij},r^*_{ij})$ to \eqref{eq:crmin} is also a feasible pair to \eqref{eq:sfspbound1}--\eqref{eq:sfsprmin1}. From \eqref{eq:crmin} and by definition of $z$ we have
\[
 z\, x^*_{ij} \le r^*_{ij}\le c_{ij}\,x^*_{ij} \implies
 \rho x^*_{ij} \leq r^*_{ij} \leq c_{ij} x^*_{ij}.
\]
This proves that $(x^*,r^*)$ satisfies \eqref{eq:sfspbound1}. To show that the inequalities \eqref{eq:sfsprmin1} are satisfied, we distinguish two cases. If $x^*_{ij} = 0$, then $r^*_{ij}=0 \implies z \leq c_{max}$, which is true by definition of $z$. If $x^*_{ij} > 0$, we have
\[
\begin{split}
z x^*_{ij} & \leq r^*_{ij}\\
z x^*_{ij} + z (1 - x^*_{ij}) & \leq r^*_{ij} + z (1 - x^*_{ij})\\
z & \leq r^*_{ij} + z (1 - x^*_{ij}) \leq r^*_{ij} + c_{max} (1 - x^*_{ij})
\end{split}
\]
where the first inequality comes from \eqref{eq:crmin}, the second inequality is true by construction, and the third inequality comes from the definition of $z$.
From this it follows that $(x^*,r^*)$ also satisfies \eqref{eq:sfsprmin1}.

\smallskip
\noindent
To show that the constraints \eqref{eq:crmin} in NLP2($z$) are strictly stronger than the constraints \eqref{eq:sfspbound1}--\eqref{eq:sfsprmin1} in NLP1(${z}$) consider an instance with $\rho=1$, $z=5$, $c_{max}=100$, $c_{ij}= 10$. One can check that the pair $(0.5,0.5)$ satisfies \eqref{eq:sfspbound1}--\eqref{eq:sfsprmin1}, but does not satisfy \eqref{eq:crmin}.
\end{proof}

\noindent
The above result is our main motivation for using the \eqref{eq:crmin}--\eqref{eq:cvarr} constraints instead of the \eqref{eq:sfspbound1}--\eqref{eq:sfsprmin1} constraints. Furthermore, as will be shown in the next section, our solution approach uses a ``non standard''  {\em Benders' decomposition} (BD), which turns out to be particularly effective when applied to the modified MINLP formulation.

\section{Benders' decomposition} \label{sec:benders}

We now present our Benders-based solution approach for the DCR problem.
%
%
Both MINLP formulations described in Section \ref{sec:formul} lend themselves well to BD since the resulting MINLPs are much easier to solve when the $z$ variable is fixed. Yet, as shown in Proposition \ref{prop:strong}, the {\em modified} MINLP formulation is stronger than the {\em initial} MISOCP, when $z$ is fixed. Therefore, we focus on the MINLP formulation \eqref{eq:of}--\eqref{eq:cvarx}, \eqref{eq:crmin}, \eqref{eq:cvarr}. It will be sometimes convenient to consider the two inequalities involved in \eqref{eq:crmin} separately:
\begin{subequations}\label{eqn:crmin-split}
 \begin{align}
  & z x_{ij} \leq r_{ij} & (i,j)\in A \label{subeqn-1:crmin-split}
    \tag{12a}\\
  & r_{ij} \leq c_{ij} x_{ij} & (i,j)\in A \label{subeqn-2:crmin-split}
    \tag{12b}
 \end{align}
\end{subequations}
A natural way to apply BD is to view $z$ as a ``complicating'' variable on which to {\em project} our problem: that is, one keeps the $z$ variable in the {\em master problem}
\begin{equation}
 \min \big\{ \, v(z) \,:\, \eqref{eq:cvarr} \, \big\}
 \label{eq:z-project-minlp}
\end{equation}
and moves all the $x$ and $r$ variables to the {\em subproblem}
\begin{equation}\label{eq:val-fun-minlp}
 \textstyle
 v(z) = \min \Big\{ \, \sum_{(i,j)\in A} {f_{ij}r_{ij}} \,:
	 \mbox{\eqref{eq:cflow}--\eqref{eq:cvarx}} \;,\;
	 \eqref{eq:crmin} \, \Big\}
\end{equation}
that computes the corresponding {\em value function} $v(z)$. Given that the subproblem is nonlinear, this requires {\em Generalized Benders Decomposition} (GBD) \cite{G1972}, that relies on Lagrangian duality to derive families of cuts corresponding to those in the BD case. In particular, if the subproblem is convex (and satisfies some regularity conditions), GBD exploits strong duality to reformulate the value function in terms of its dual representation. The key idea of GBD is to ``dualize'' the linking constraints between the complicating variables and the other variables. Besides convexity, GBD also requires other strong structural properties of the linking constraints, the simplest one being {\em separability}. For brevity, we do not describe this approach in detail since, in our case, the subproblem lacks both separability and convexity, hence the standard GBD strategy cannot be directly applied: rather, the underlying idea of using Lagrangian relaxation inside the Benders' approach is played in a creative way.

\subsection{A nested Benders-Lagrange approach}\label{ssec:quasiBend}

In this section we describe an alternative BD approach for the modified MINLP formulation \eqref{eq:of}--\eqref{eq:cvarx}, \eqref{eq:crmin}, \eqref{eq:cvarr}.
Instead of following the classical GBD strategy, we dualize the constraints \eqref{eq:cflow} and \eqref{eq:cdelay} with Lagrangian multipliers $[\,\pi_i\,]_{i \in N}$ and $\lambda \geq 0$, respectively, to obtain:
\begin{align}
 \psi(\pi, \lambda; z) =
 & \textstyle
   \min \Big\{ \sum_{(i,j) \in A} f_{ij}r_{ij} +
               \lambda \big( \frac {L {x_{ij}}^2}{r_{ij}} + {\bar l}_{ij} x_{ij} \big) + (\pi_j - \pi_i) x_{ij}
               \,:\, \eqref{eq:cvarx} \,,\, \eqref{eq:crmin} \, \Big\}
   \nonumber \\
 & \textstyle
   - \sum_{i \in N} \pi_i b_i -\lambda \delta + \frac{\lambda \sigma}{z}
 \label{eq:Lagfix-minlp-alt}
\end{align}
When $z$ is fixed, the expression in \eqref{eq:Lagfix-minlp-alt} as a function of $(\pi, \lambda)$ represents the Lagrangian function, and
\begin{equation}
 \underline{v}(z) =
 \max \big\{ \, \psi(\pi, \lambda; z) \,:\, \pi \in \RR^{|N|} \;,\;
                \lambda \geq 0 \, \big\}
\label{eq:LagDfix-minlp-alt}
\end{equation}
the corresponding Lagrangian dual. Note, however, that the $x$ variables in \eqref{eq:Lagfix-minlp-alt} are binary, making the subproblem nonconvex. Hence, by constructing the Lagrangian dual, we obtain a {\em convex relaxation} of \eqref{eq:val-fun-minlp}. In particular, \eqref{eq:LagDfix-minlp-alt} gives a {\em lower approximation} of the original value function $v(z)$.  Hence, by solving the {\em relaxed master problem}
\begin{equation}
	\min \{ \, {\underline v}(z) \,:\, \eqref{eq:cvarr} \, \}
	\label{eq:z-project-minlp-alt-rel}
\end{equation}
we obtain a lower bound on the optimal value of our problem.
Owing to the bilinear term in \eqref{subeqn-1:crmin-split}, the minimization problem in \eqref{eq:Lagfix-minlp-alt} still lacks separability with respect to $z$, yet the advantage of this relaxation is twofold:
\begin{itemize}
 \item the Lagrangian dual \eqref{eq:LagDfix-minlp-alt} can be efficiently solved;
 \item despite being nonconvex, \eqref{eq:Lagfix-minlp-alt} has a ``nice'' structure that allows to construct piecewise linear (lower) approximations for it, that can be used as Benders' cuts.
\end{itemize}
In the next sections we describe how to exploit these two features.

\section{Solving the Lagrangian dual}\label{sec:lagsolve}

The approach is based on the fact that the only multiplier that is really nontrivial to find is $\lambda^*$, i.e., the optimal solution of the \emph{univariate} Lagrangian dual
\begin{equation}
 \underline{v}_1(z) =
 \max \{ \, \psi_{1}(\lambda; z) \,:\, \lambda \geq 0 \, \}
 \label{eq:LagDfix-minlp-alt-bis}
\end{equation}
with
\begin{align}
  \psi_1(\lambda; z) = &
  \textstyle
  \min \Big\{ \, \sum_{(i,j)\in A} {f_{ij}r_{ij}} +
                 \lambda \big( \frac {L {x_{ij}}^2}{r_{ij}}
                                + {\bar l}_{ij} x_{ij} \big)
              \;:\; \eqref{eq:cflow} \,,\, \eqref{eq:cvarx} \,,\,
                    \eqref{eq:crmin} \, \Big\}
  \nonumber \\
  & \textstyle
    -\lambda \delta + \frac{\lambda \sigma}{z}
    \label{eq:Lagfix-minlp-alt-bis}
\end{align}
The important property of \eqref{eq:Lagfix-minlp-alt-bis} is that, when $\lambda $ is fixed, we can ``project out'' the $r_{ij}$ variables onto the $x$--space. If $x_{ij}=0$, the corresponding reserved rate $r_{ij}$ is set to zero. If $x_{ij}=1$, the corresponding optimal ${\bar r}_{ij}$ is given by the one-dimensional convex optimization problem
\begin{equation}
 \textstyle
 \min \Big\{ \, h(r_{ij}) = f_{ij}r_{ij} +
                \lambda \big( \frac{L}{r_{ij}} + {\bar l}_{ij} \big)
             \;:\; z \leq r_{ij} \leq c_{ij} \, \Big\}
 \label{eq:rijproject}
\end{equation}
which, from ordinary first-order calculus {(and assuming $f_{ij} \geq 0$)},  gives
\begin{equation}
 \textstyle
 {\bar r}_{ij}=
 \min \Big\{ \, c_{ij} \,,\,
                \max\big\{ \, z \,,\, {\bar v} =
                                      \sqrt{\lambda L \,/\, f_{ij}}
                           \, \big\} \, \Big\}
 \label{SPopt}
\end{equation}
Projection onto the $x$--space allows us to reformulate \eqref{eq:Lagfix-minlp-alt-bis} as a {\em Shortest Path Problem} (SPP)
\begin{equation}
 \textstyle
 - {\lambda} \delta + \frac{{\lambda} \sigma}{z}
 + \min \big\{  \sum_{(i,j)\in A}\tilde{c}_{ij} x_{ij} :
                  \eqref{eq:cflow} \,,\, \eqref{eq:cvarx} \big\}
 \label{eq:SP}
\end{equation}
with {non-negative arc costs} $\tilde{c}_{ij}= h({\bar r}_{ij})$. Since $\psi_{1}(\lambda; z)$ can be computed efficiently for any value of ${\lambda}$, the Lagrangian dual \eqref{eq:LagDfix-minlp-alt-bis} can also be solved efficiently. In particular, \eqref{eq:LagDfix-minlp-alt-bis} is a one-dimensional optimization problem whose objective to be maximized is a concave non-differentiable function, so we can use a {\em line search} as an adaptation of the classical Kelley's Cutting Plane algorithm \cite{K1960} to the one-dimensional case. For any $\lambda \geq 0$, we denote by $\bigl(x(\lambda)\,,\, r(\lambda)\bigr)$ the corresponding optimal solution to the Lagrangian problem in \eqref{eq:Lagfix-minlp-alt-bis}: the subgradient of $\psi_1(\lambda; z)$ in $\lambda$ is
\begin{equation}
 \textstyle
 s\bigl( x(\lambda) \,,\, r(\lambda) \bigr) =
 \frac{\sigma}{z} -\delta +
 \sum_{(i,j)\in A} \frac{Lx^2_{ij}}{r_{ij}} + \bar{l}_{ij} x_{ij}
 \in \partial \psi_1(\lambda; z) \; .
 \label{eq:grad}
\end{equation}
At each iteration of the {\em line search}, we need two values of $\lambda$, $0 \leq \lambda^+ < \lambda^-$, such that $s( x(\lambda^+) \,,\, r(\lambda^+)) > 0$ and $s( x(\lambda^-) \,,\, r(\lambda^-) ) < 0$. The two subgradients are used, respectively, to construct two lines $\ell^+(\lambda)$ and $\ell^-(\lambda)$, which {\em upper} approximate the Lagrangian function $\psi_{1}(\lambda; z)$. Next, we compute the point $\hat{\lambda}$ corresponding to the intersection of the two lines and the corresponding solution $\bigl(x(\hat{\lambda})$, $r(\hat{\lambda})\bigr)$, which either replaces  $\lambda^-$ or $\lambda^+$, depending on the sign of the subgradient $s\bigl(x(\hat{\lambda}) \,,\, r(\hat{\lambda}) \bigr)$. The process is iterated until $\ell^+({\hat \lambda}) = \ell^-({\hat \lambda}) = \psi_{1}({\hat \lambda}, z)$, which implies optimality of ${\hat \lambda} = \lambda^*$. The approach is completely standard (cf.~e.g.~\cite{FGS2024} and the references therein) and we omit the details for brevity. Note that anytime we find a multiplier $\lambda^-$ with negative subgradient, the corresponding constraint \eqref{eq:cdelay} is satisfied in $( \, x(\lambda^-) \,,\, r(\lambda^-) \, )$, and thus represents a feasible solution for our problem. Hence, the solution of the Lagrangian dual \eqref{eq:LagDfix-minlp-alt-bis} also doubles as a matheuristic that produces upper bounds.

\subsection{Computing optimal primal and dual solutions}\label{ssec:kkt}

The crucial issue for the application of Benders' approach is the ability to compute the dual optimal solutions that enter into the Benders' cuts. This requires first and foremost to properly identify the \emph{convexified relaxation} of the original problem \cite{LR01,F2005} that the Lagrangian dual solves. We will now detail how to find, by means of a convex combination of the (two) integer solutions at $\lambda^*$, the optimal \emph{primal} solution of that problem, on our route for the crucial optimal \emph{dual} solution.

\smallskip
\noindent
Let $\xr{1}, \xr{2}$ be the two solutions at the last iteration of the line search, i.e., two different optimal solutions to $\psi_1(\lambda^*; z)$ such that the corresponding sub-gradients (with opposite sign) $s\xr{1},\, s\xr{2}$ both belong to $\partial \psi_1(\lambda^*; z)$. We define $\xr{\nu} = \nu \xr{1} + ( 1 - \nu ) \xr{2}$ their convex combination for some $\nu \in [0,1]$. By optimality of $\lambda^*$ we have
\[
 0 \in \partial \psi_1(\lambda^*; z)
 \implies \exists\,\, \nu^* \in [0,1] \;:\;
 \nu^* s\xr{1} + (1 - \nu^*) s\xr{2} = 0 \;.
\]
One would therefore naturally look at the solution
\begin{equation}\label{eq:xtilde}
 (\tilde x, \tilde r) = \xr{\nu^*} = \nu^* \xr{1}+( 1 - \nu^* )\xr{2}
\end{equation}
as the candidate solution for the convexified relaxation, that we denote by $\tilde P(z)$. This would certainly be the right answer if the relaxed constraints were linear \cite{F2005}, but in our case they are not and the answer is generally more complex \cite{LR01}. Yet, we can show that the result actually holds, starting from the following Lemma that characterizes optimal solutions to the Lagrangian relaxation $\psi_1(\lambda; z)$ for any given $\lambda$ (and therefore, in particular, to $\xr{1}$ and $\xr{2}$):

\begin{lemma}\label{lem:struct}
Any optimal solution $\xr{h}$ to $\psi_1(\lambda; z)$ satisfies $r^h = x^h \cdot {\bar r}.$
\end{lemma}
\begin{proof}
If $x^h_{ij} = 0$, constraints \eqref{eq:crmin} imply $r^h_{ij} = 0$.
When $x^h_{ij} = 1$, since the optimal value of \eqref{eq:rijproject} only depends on  $\lambda$, then $r^h_{ij} = {\bar r}_{ij}$.
\end{proof}

\noindent
Due to this ``special'' structure, the sub-gradient function \eqref{eq:grad} \emph{behaves linearly} with respect to convex combinations of $\xr{1}$ and $\xr{2}$. To prove it, we start by constructing a partition of the arc set $A$:
\begin{align*}
 & A_1 := \{\, (i,j)\in A \ :\ x^1_{ij} = 1 \,,\, x^2_{ij} = 0 \, \}\\
 & A_2 := \{\, (i,j)\in A \ :\ x^1_{ij} = 0 \,,\, x^2_{ij} = 1 \, \}\\
 & A_{\cap} := \{ \, (i,j)\in A \ :\ x^1_{ij} = x^2_{ij} = 1 \, \}\\
 & A_{\emptyset} := \{ \, (i,j)\in A \ :\ x^1_{ij} = x^2_{ij} = 0 \, \}
\end{align*}
and we define the function
\[
 \textstyle
 F_{A'}(x,r) =
 \sum_{(i,j)\in A'} \frac{Lx_{ij}^2}{r_{ij}}+  \bar{l}_{ij} x_{ij}
\]
representing the delay contribution of a subset of arcs $A' \subseteq A$ for a given solution $(x,r)$. The following proposition shows the ``linear behavior'' of $F_{A}(x,r)$.

\begin{proposition}\label{prop:Fcc}
The function $F_{A}(x,r)$ behaves in a linear way with respect to convex combinations of $\xr{1}, \xr{2}$.
\end{proposition}
\begin{proof}
Let $F_A\xr{\nu} = \nu F_A \xr{1} + (1-\nu) F_A \xr{2}$. From Lemma \ref{lem:struct} we have that:
\begin{align*}
 &F_{A_{\emptyset}}\xr{\nu} = F_{A_{\emptyset}}\xr{1}= F_{A_{\emptyset}}\xr{2} = 0\\
 &F_{A_{1}}\xr{\nu} = \nu F_{A_1}\xr{1}\\
 &F_{A_{2}}\xr{\nu} = (1-\nu) F_{A_2}\xr{2}\\
 &F_{A_{\cap}}\xr{\nu} = F_{A_{\cap}}\xr{1} = F_{A_{\cap}}\xr{2} = \nu F_{A_{\cap}}\xr{1} + (1-\nu)F_{A_{\cap}}\xr{2}.
\end{align*}
Thus,
\begin{align*}
 F_{A}\xr{\nu}
 & = F_{A_{\emptyset}}\xr{\nu}+F_{A_{1}}\xr{\nu}+F_{A_{2}}\xr{\nu}+ F_{A_{\cap}}\xr{\nu}\\
 & =\nu F_{A_1}\xr{1} + (1-\nu) F_{A_2}\xr{2}+ \nu F_{A_{\cap}}\xr{1} + (1-\nu)F_{A_{\cap}}\xr{2} \\
 & = \nu F_A \xr{1} + (1-\nu) F_A \xr{2}. \\[-1.2cm]
\end{align*}
\end{proof}

\noindent
As an immediate consequence:

\begin{corollary}\label{cor:gcc}
The sub-gradient function $s(x,r)$ behaves in a linear way with respect to convex combinations of $\xr{1}, \xr{2}$: for any $\nu \in [0,1]$
\[
 \nu s\xr{1} + (1 - \nu) s\xr{2}
 = s\big( \nu \xr{1} + (1 - \nu) \xr{2} \big) = s \xr{\nu}
 \;
 \]
\end{corollary}
\begin{proof}
The result follows directly from Proposition \ref{prop:Fcc} and observing that, according to \eqref{eq:grad}, $s(x,r) = \frac{\sigma}{z} -\delta + F_{A}(x,r)$.
\end{proof}

All this finally leads to the announced result:

\begin{proposition}\label{prop:solPt}
The solution $(\tilde x, \tilde r)$ defined in \eqref{eq:xtilde} is optimal for $\tilde P(z)$.
\end{proposition}
\begin{proof}
The solution $(\tilde x, \tilde r)$ is optimal for $\psi_1(\lambda^*; z)$ being a convex combination of two optimal solutions $\xr{1}$ and $\xr{2}$. By Corollary \ref{cor:gcc}, $(\tilde x, \tilde r)$ satisfies the complementary slackness condition $s(\tilde x, \tilde r)= 0$. Hence, it is feasible and optimal for the (primal) convex relaxation $\tilde P(z)$.
\end{proof}

\noindent
With all this machinery, we are now able to characterise the (rest of the) optimal solution to \eqref{eq:LagDfix-minlp-alt} (besides of course $\lambda^*$). We denote by $\bar P(z)$ the \emph{continuous relaxation} of the modified MINLP (when $z$ is fixed), i.e., \eqref{eq:of}--\eqref{eq:cdelay}, \eqref{eq:crmin} and
\begin{equation}
 0 \leq x_{ij} \leq 1 \qquad (i,j)\in A \label{eq:c7} \tag{c7}
\end{equation}
in lieu of \eqref{eq:cvarx}. We will now show that the optimal node prices $[d^*_i]_{i \in N}$  of the SPP($\lambda^*$) problem \eqref{eq:SP} solved for $\lambda = \lambda^*$ are all that we need to compute the sought-after complete dual solution. As a by-product, this shows that $\bar P(z)$ is actually equivalent to the convexified relaxation $\tilde P(z)$ (a result that would have been obvious for a problem with linear constraints, but it is not here).

\smallskip
\noindent
Clearly, the optimal (primal) solution $(\tilde x, \tilde r)$ to $\tilde P(z)$ is feasible for $\bar P(z)$, by inclusion of the corresponding feasible regions. We now show that $(\tilde x, \tilde r)$ is also optimal for $\bar P(z)$ by constructing a dual solution that, together with $(\tilde x, \tilde r)$, satisfies the Karush–Kuhn–Tucker (KKT) conditions. We will use the following notation to denote the {\em dual variables}:
\begin{itemize}
 \item $[\pi_i]_{i \in N}$ are associated to constraints \eqref{eq:cflow};
 \item $\lambda$ is  associated to constraint \eqref{eq:cdelay};
 \item $\eta^+_{ij}$ and $\eta^-_{ij}$, $(i,j) \in A$, are  associated to $\leq$ and $\geq$ constraints in \eqref{eq:c7}, respectively;
 \item $\mu^+_{ij}$ and $\mu^-_{ij}$, $(i,j) \in A$, are  associated to $\leq$ and $\geq$ constraints in \eqref{eq:crmin}, respectively.
\end{itemize}
The KKT conditions require {\em stationarity}
\begin{align}
 & \textstyle
   2\lambda L{\frac{x_{ij}} {r_{ij}}} + \lambda\bar{l}_{ij} -
   \pi_i+\pi_j-c_{ij}\mu^+_{ij}+z\mu^-_{ij}-\eta^-_{ij}+\eta^+_{ij} = 0
 & (i,j)\in A \label{eq:st1} \tag{st1}\\
 & \textstyle
   f_{ij} - \lambda L\left(\frac{x_{ij}}{r_{ij}}\right)^2
   +\mu^+_{ij}-\mu^-_{ij} = 0
 & (i,j)\in A \label{eq:st2} \tag{st2}
\end{align}
{\em primal and dual feasibility}
\begin{align}
 & \mu_{ij}^+\ge 0   &(i,j)\in A \label{eq:df1} \tag{df1} \\
 & \mu_{ij}^-\ge0    &(i,j)\in A \label{eq:df2} \tag{df2} \\
 & \lambda \ge 0     &(i,j)\in A \label{eq:df3} \tag{df3} \\
 & \eta_{ij}^+\ge 0  &(i,j)\in A \label{eq:df4} \tag{df4} \\
 & \eta_{ij}^-\ge0   &(i,j)\in A \label{eq:df5} \tag{df5}
\end{align}
and {\em complementary slackness}:
\begin{align}
 & \textstyle
   \lambda \big({\frac{\sigma}{ z}} - \delta +
   \sum_{(i,j)\in A}  \frac{Lx_{ij}^2}{r_{ij}}-\bar{l}_{ij} x_{ij}\big) = 0
 &  \label{eq:cs1} \tag{cs1}\\
 & \mu^+_{ij}(r_{ij}-c_{ij}x_{ij})=0  &(i,j)\in A \label{eq:cs2} \tag{cs2}\\
 & \mu^-_{ij}(x_{ij}z-r_{ij})=0    &(i,j)\in A \label{eq:cs3} \tag{cs3}\\
 & -\eta^-_{ij}x_{ij}=0   &(i,j)\in A \label{eq:cs4} \tag{cs4}\\
 & \eta^+_{ij}(x_{ij}-1)=0  &(i,j)\in A \label{eq:cs5} \tag{cs5}
\end{align}
Note that the solution $(\tilde x, \tilde r)$ satisfies primal feasibility with respect to  $\bar P(z)$ and the complementary slackness condition \eqref{eq:cs1}. In order to construct a dual solution $({\tilde \pi}, {\tilde \lambda}, {\tilde \eta}^+, {\tilde \eta}^-, {\tilde \mu}^+, {\tilde \mu}^-)$ that meets all the KKT conditions with $(\tilde x, \tilde r)$ we consider the possible values of ${\tilde x}_{ij}$ for each arc $(i,j) \in A$.

\subsubsection{Case ${\tilde x}_{ij} \neq 0$}\label{subseq:casexstar1}

If ${\tilde x}_{ij} = 1$, then $x^1_{ij} = x^2_{ij} = 1$, and ${\tilde r}_{ij} = r^1_{ij} = r^2_{ij} = {\bar r}_{ij}$. We set  ${\tilde \eta}^+_{ij} = {\tilde \eta}^-_{ij}= 0$, ${\tilde \pi} = - d^*$ and ${\tilde \lambda} = \lambda^*$. The value of the remaining dual variables depends on ${\bar r}_{ij}$ as follows:
\begin{equation}
 \label{eq:kktsol}
 ({\tilde \mu}^+_{ij},{\tilde \mu}^-_{ij}) =
 \begin{cases}
  (0,f_{ij}-\frac{\lambda^* L}{z^2}) & \text{if} \ {\bar r}_{ij} = z,\\
  (0,0) & \text{if} \ {\bar r}_{ij} = {\bar v},\\
  (-f_{ij}+\frac{\lambda^* L}{c_{ij}^2},0) & \text{if} \ {\bar r}_{ij}= c_{ij} .
 \end{cases}
\end{equation}
By setting ${\tilde \pi} = - d^*$, we have $\tilde{c_{ij}}-\pi_i+\pi_j=0$, since the vector $d^*$ meets the Bellman conditions for the SPP \eqref{eq:SP} with non-negative arc costs $\tilde{c}_{ij}$. Note also that  ${\bar v}$ is a stationary point for $h(r_{ij})$ (defined in \eqref{eq:rijproject}), and the assignment \eqref{eq:kktsol} guarantees dual feasibility with respect to \eqref{eq:df1}--\eqref{eq:df2} according to \eqref {SPopt}. With this in mind, one can check that all the KKT conditions are satisfied. Indeed, if $0< {\tilde x}_{ij}< 1$, then exactly one between $x^1_{ij}$ and $x^2_{ij}$ is equal to 1. We assume, without loss of generality, that $x^1_{ij} = 1$. This implies ${\tilde x}_{ij} = \nu^* x^1_{ij}$ and ${\tilde r}_{ij} = \nu^* r^1_{ij}$. Anyway, the coefficient $\nu^*$ cancels out, so the dual values defined for the case ${\tilde x}_{ij} = 1$ still satisfy the KKT conditions since we apply the same reasoning to $x^1_{ij} = 1$.

\subsubsection{Case $\tilde x_{ij} = 0$}

If ${\tilde x}_{ij} = 0$, then $x^1_{ij} = x^2_{ij} = 0$ and ${\tilde r}_{ij} = 0$. Hence,  each term $h(x,r) = x^2 / r$ in \eqref{eq:cdelay} (omitting the subscripts $(i,j)$ to ease notation) is non-differentiable. To get around this, it is enough to find a subgradient in $\partial h(0,0)$  meeting the KKT conditions.
In particular, we define $f(r) = 1 / r$ and observe that $h(x,r)$ is the {\em perspective function} of $f(r)$ for $xz\le r \le x c_{ij}$ (see constraints \eqref{eq:crmin}) and $x \in [0,1]$. It is well-known that the perspective function can be extended by continuity in $(0,0)$, which gives $h(0,0)= 0$. Using \cite[(5)]{FG2006}, we have
\begin{equation}
 \textstyle
 \big[\nabla f ( \frac{{x}}{{r}}) \,,\,
      f( \frac{{x}}{{r}} ) -
      \nabla f( \frac{{x}}{{r}} )\frac{{x}}{{r}} \big] =
 \big[ - \left(\frac{{x}}{{r}}\right)^2 \,,\, 2\frac{{x}}{{r}} \big]
 \in \partial {h}(0,0) \;, \label{eq:FGsub}
\end{equation}
and substituting \eqref{eq:FGsub} to construct the stationary KKT conditions, we obtain the same equations \eqref{eq:st1}--\eqref{eq:st2}. This shows that using the Perspective Reformulation is crucial to obtain a tight formulation---and therefore the necessary optimal dual solution; this is intuitive, and has been already shown elsewhere \cite{BFGTG24}. The rest of the KKT conditions are unchanged.

\smallskip
\noindent
We are now ready to construct the dual solution when $\tilde{x}_{ij}=\tilde{r}_{ij} = 0$. To be consistent with the dual solution constructed in  \S \ref{subseq:casexstar1}, we set $\tilde \pi = -d^*$ and $\tilde \lambda = \lambda^*$. To meet \eqref{eq:cs5} we set $\tilde  \eta^+_{ij} = 0$, so the resulting system \eqref{eq:st1}--\eqref{eq:st2} becomes:
\begin{equation*}
\begin{cases}
 \lambda^*\bar{l}_{ij} + d^*_i - d^*_j +
 z\mu^-_{ij} - c_{ij}\mu^+_{ij} - \eta^-_{ij} = 0, \\
 f_{ij} + \mu^+_{ij} - \mu^-_{ij} = 0.
\end{cases}
\end{equation*}
from which we get $\mu^-_{ij}=f_{ij}+\mu^+_{ij}$ and
\[
 \lambda^*\bar{l}_{ij} + zf_{ij}+ d^*_i -d^*_j
 +(z-c_{ij})\mu^+_{ij}-\eta^-_{ij}=0 \;.
\]
Denoting by $q_{ij} = \lambda^*\bar{l}_{ij} + zf_{ij}+ d^*_i -d^*_j$, we can set
\begin{equation}
 (\mu^+_{ij},\eta^-_{ij}) =
 \begin{cases}
  (0,q_{ij}) & \text{if} \ q> 0\\
  (-\frac{q}{z-c_{ij}},0) & \text{if} \ q< 0\\
  (0,0) & \text{if} \ q=0
 \end{cases}
\end{equation}
meeting the KKT conditions. Note that this construction requires that the node prices $d^*$ be defined. However, if a node $i$ is unreachable from the source $s$, the corresponding price $d^*_i$ is undefined (i.e., $d^*_i = +\infty$).
Yet, one can easily show that it is still possible to construct a dual solution satisfying the KKT conditions, by similar reasoning. We omit the details for brevity.

\subsubsection{Integrality property}

We constructed a pair $(\tilde x, \tilde r)$ and $({\tilde \pi}, {\tilde \lambda}, {\tilde \eta}^+, {\tilde \eta}^-, {\tilde \mu}^+, {\tilde \mu}^-)$ of primal and dual solutions for $\bar P(z)$, respectively, that satisfy the KKT conditions.
Since $\bar P(z)$ is convex, this proves that such solutions are optimal and {\em strong duality} holds for $\bar P(z)$. Note that it also shows that the relaxed problems $\bar P(z)$ and $\tilde P(z)$ are equivalent. This fact can be interpreted as a sort of ``integrality property'' (originally defined for linear problems, see \cite{F2005}) of the (nonlinear) Lagrangian $\psi_{1}(\lambda; z)$, meaning that relaxing integrality does not affect the corresponding value function  $\underline{v}_{1}(z)$. This in turn implies that the integrality property also applies to the Lagrangian $\psi(\pi, \lambda; z)$, for which the multipliers $({\tilde \pi}, {\tilde \lambda})$ are optimal. Therefore, by (efficiently) solving the (univariate) Lagrangian dual \eqref{eq:LagDfix-minlp-alt-bis} we obtain (for free) the optimal dual multipliers ${\tilde \pi} = - d^*$, ${\tilde \lambda} = \lambda^*$ of the Lagragian dual \eqref{eq:LagDfix-minlp-alt}, as desired, and the corresponding value functions $\underline{v}_1(z)$ and  $\underline{v}(z)$ are equivalent. Remarkably, this hinges on two non-obvious modelling choices: the use of the Perspective Reformulation, and the nonconvex constraints \eqref{subeqn-1:crmin-split}. Without both, the formulation would be weaker, and therefore we would not have a primal/dual optimal pair with which to construct an approximation of $\underline{v}_1(z)$: this is done right next.

\smallskip
\noindent
It has to be remarked that the integrality property obviously does not imply that we are solving the problem with fixed $z$ to optimality (DCR generalises the constrained shortest path problem and it is therefore in general $\mathcal{NP}$-hard \cite{FGS2015a}). Indeed, the optimal primal (convexified) solution of the Lagrangian dual is the convex combination of two paths, and therefore in general fractional. Hence, $\underline{v}_1(z)$ is in general only a lower bound on $v(z)$, the interesting question computationally being how strong it is.

\section{Approximating the value function}\label{sec:approx}

Once $v(z)$ has been computed with the approach of the previous section, we can construct the announced (lower) piecewise linear approximation of $\underline{v}(\cdot)$ (equivalently, $\underline{v}_1(\cdot)$) in $z$. Using \eqref{eq:LagDfix-minlp-alt}, the Benders reformulation is
\begin{align}
	\min \, & z_0 & \label{eq:ofB} \\
	& z_0 \geq \psi(\lambda, \pi; z)
	& \lambda \geq 0 \;,\; \pi \in \RR^{|N|}  \label{eq:GBcuts} \\
	& z \in [\, \rho \,,\, c_{max} \,] & \label{eq:cvarrz}
\end{align}
As customary, the infinitely many \eqref{eq:GBcuts} are sampled by fixing some $z = {\bar z}$ and using the optimal values $\lambda^*({\bar z})$ and $\pi^*({\bar z})$ computed as above, since they maximize the value. This produces a pointwise lower-approximations to \eqref{eq:GBcuts}; however, in order to efficiently solve the minimization problem, one has to develop \emph{Benders' cuts}, i.e., (linear) lower approximations of $\underline{v}(z) = \psi(\lambda^*(z),\pi^*(z); z)$ that drive the search for the optimal $z^*$. Unlike in the standard case, $\underline{v}(z)$---besides being in general nondifferentiable due to its max-nature---is in general a nonconvex function of $z$, which means that a single linear approximation may not be valid for all the domain $[\, \rho \,,\, c_{max} \,]$. The crux is therefore to carefully analyse $\underline{v}(\cdot)$ in order to derive multiple valid approximations at some point $z$ out of the available solution information.

\smallskip
\noindent
As customary for a max-function, the first-order information in $z$ is obtained by considering $\lambda$ and $\pi$ fixed to the values $\lambda^*(z)$ and $\pi^*(z)$ obtained for the given $z$ and analysing how $\psi(\lambda,\pi; z)$ behaves as a function of $z$. In this sense we distinguish the \emph{constant} term $-\lambda \delta - \sum_{i \in N} \pi_i b_i$ and the \emph{convex} term $\lambda \sigma / z$, for which computing linear lower approximations is trivial, from the \emph{nonconvex} term corresponding to the optimization problem
\begin{equation}
 \textstyle
 \psi(z) =
 \min \Big\{ \sum_{(i,j)\in A} f_{ij}r_{ij} +
             \frac{\lambda L{x_{ij}}^2}{r_{ij}} +
             (\lambda \bar{l}_{ij} + \pi_j - \pi_i) x_{ij} \,:\,
             \eqref{eq:cvarx}, \eqref{eq:crmin} \Big\}
 \label{eq:gopt}
\end{equation}
The first observation is that \eqref{eq:gopt} is separable with respect to the arcs $(i,j) \in A$, i.e., $\psi(z) = \sum_{(i,j)\in A} \phi_{ij}(z)$. Clearly, we can then focus on each single term individually: denoting by ${\bar d}_{ij} = \pi_j - \pi_i$, and dropping the subscript ${(i,j)}$ to ease notation, we get the (crucial) nonconvex function
\begin{equation}
 \textstyle
 \phi(z) =
 \min \big\{ f r + \frac{\lambda Lx^2}{r} + (\lambda \bar{l} + {\bar d}) x
             \,:\, z\, x \le r \le c \,x \;,\; x \in \{0,1\} \big\}
 \label{eq:fi}
\end{equation}

\subsection{Analysis of $\phi(z)$}

As usual, in the following, we denote by $\bar z$ the value of $z$ for which $\lambda$ and $\pi$ were computing according to \eqref{eq:LagDfix-minlp-alt}.
Since the variable $x$ in \eqref{eq:fi} is binary, we need only consider two cases in our analysis of $\phi(z)$. The simplest is $x = 0$ ($\implies r=0$), where $\phi(z) = 0$. In particular, this surely happens if $z > c$, as $x = 0$ is the only feasible choice. Note that, hence, $\phi(z) \leq 0$ for all $z$ since $\phi(\cdot)$ is defined by a minimum. The interesting case is then $x = 1$, where $\phi(z) = h(r(z)) + {\bar d}$ with $h(r) = f r + \lambda (L / r + \bar{l})$ and
\[
 r(z) = \min\{ \, h(r) \,:\, z \le r \le c \, \}
      =  \min\{ \, c \,,\, \max\{ \, \bar{v} \,,\, z \, \} \, \}
\]
according to \eqref{eq:rijproject} and \eqref{SPopt}. Indeed, the constant $\bar d$ does not change the point ${r(z)}$ where $h(r)$ attains the minimum. We start proving the following property:

\begin{proposition}\label{prop:zbar}
 $\phi(z) = 0$ for all $z \geq {\bar z}$.
\end{proposition}
\begin{proof}
The Bellman conditions guarantee that any (dual) optimal solution $d^*$ to the Shortest Path problem  \eqref{eq:SP} (for the given ${\bar z}$) satisfies ${\tilde c}_{ij} \geq d^*_j - d^*_i$ for each arc $(i,j) \in A$, and hence ${\tilde c} \geq -{\bar d}$. By definition, ${\tilde c} = h(r({\bar z}))$, whence
\[
 h(r(z)) + {\bar d} \geq h(r(z)) - h(r({\bar z}))
 \implies
 0 \geq \phi(r(z)) = h(r(z)) + {\bar d} \geq 0 \;.
 \vspace{-0.8cm}
\]
\end{proof}

\noindent
All in all, $x=1$ can only happen for the values of $z$ for which $h(r(z)) +  {\bar d} < 0$. Note that the objective in the minimization of \eqref{eq:fi} is increasing in $\lambda \geq 0$ (both $Lx^2 / r \geq 0$ and $\bar{l} \geq 0$), so it is convenient to examine separately the edge case $\lambda = 0$ from $\lambda > 0$.

\subsubsection{The case $\lambda = 0$}\label{sssec:philam0}

The objective in the minimization of \eqref{eq:fi} is the line $f r + {\bar d}$, with positive slope ($f \geq 0$), with a zero at $p=-{\bar d}/f$. Clearly, if $p \leq 0$, the function $\phi(z)$ is identically zero for all values of $z$: the objective is non-negative on all feasible $0 \leq \rho \leq r$, and therefore $x = 0$ is always chosen. Therefore we only consider the $p > 0$ case, which also implies $\bar{d} < 0$. Then, $\phi(z)$ is  described by
\begin{equation}
 \phi(z) =
 \begin{cases}
  fz + {\bar d} &\text{if} \ \ z\le\min\{p,c\}\\
  0             &\text{otherwise}
 \end{cases}
 \; .
 \label{eq:phi2lin}
\end{equation}
The definition in \eqref{eq:phi2lin} preserves continuity if $\min\{p,c\}=p$, as shown in Figure~\ref{fig:phi-c-min}(left); otherwise, it produces a jump discontinuity at $z=c$, from the negative value $fc+\bar d$ to $0$, as shown in Figure~\ref{fig:phi-c-min}(right).

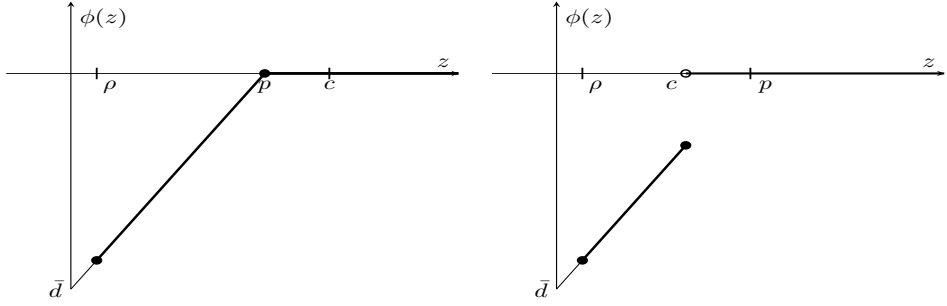
\begin{figure}[htbp]
\resizebox{\textwidth}{4cm}{
\begin{tikzpicture}
\begin{axis}[
    axis lines=middle,
    axis line style={thin},
    xlabel={$z$},
    ylabel={$\phi(z)$},
    xmin=-1, xmax=6,
    ymin=-3, ymax=1,
    xtick=\empty,
    ytick=\empty,
    clip=false
]

\pgfmathsetmacro{\P}{3}       
\pgfmathsetmacro{\Rho}{0.4}   
\pgfmathsetmacro{\C}{4}       
\pgfmathsetmacro{\F}{1}       
\pgfmathsetmacro{\D}{-3}      

\pgfmathsetmacro{\Yr}{\F*\Rho+\D}
\pgfmathsetmacro{\Yp}{\F*\P+\D}

\addplot[
    thin,
    domain=0:\Rho,
    samples=100
]
{\F*x + \D};

\addplot[
    very thick,
    domain=\Rho:\P,
    samples=100
]
{\F*x + \D};


\draw[
    line width=1.6pt
]
(axis cs:\P,0) -- (axis cs:6,0);

\draw[thick]
    (axis cs:\Rho,-0.08) --
    (axis cs:\Rho,0.08);
coordinates {(\Rho,0)};

\addplot[
    only marks,
    mark=*,
    thick
]
coordinates {(\Rho,\Yr)};

\node[below] at (axis cs:\Rho+0.2,0) {$\rho$};

\addplot[
    only marks,
    mark=*,
    thick
]
coordinates {(\P,0)};

\node[below] at (axis cs:\P,0) {$p$};

\draw[thick]
    (axis cs:\C,-0.08) --
    (axis cs:\C,0.08);
coordinates {(\C,0)};

\node[below] at (axis cs:\C,0) {$c$};

\node[left] at (axis cs:0,\D) {$\bar d$};

\end{axis}
\end{tikzpicture}
\;\;
\begin{tikzpicture}
\begin{axis}[
    axis lines=middle,
    xlabel={$z$},
    ylabel={$\phi(z)$},
    xmin=-1, xmax=6,
    ymin=-3, ymax=1,
    xtick=\empty,
    ytick=\empty,
    clip=false
]

\pgfmathsetmacro{\F}{1}
\pgfmathsetmacro{\D}{-3}
\pgfmathsetmacro{\P}{-\D/\F}   
\pgfmathsetmacro{\C}{2}        
\pgfmathsetmacro{\Rho}{0.4}    

\pgfmathsetmacro{\YC}{\F*\C+\D}
\pgfmathsetmacro{\Yr}{\F*\Rho+\D}

\addplot[
    thin,
    domain=0:\Rho,
    samples=100
]
{\F*x + \D};

\addplot[
    very thick,
    domain=\Rho:\C,
    samples=100
]
{\F*x + \D};

\addplot[
    very thick,
    domain=\C:6,
    samples=2
]
{0};


\draw[thick]
    (axis cs:\Rho,-0.08) --
    (axis cs:\Rho,0.08);
coordinates {(\Rho,0)};

\addplot[
    only marks,
    mark=*,
    thick
]
coordinates {(\Rho,\Yr)};

\node[below] at (axis cs:\Rho+0.2,0) {$\rho$};

\addplot[
    only marks,
    mark=*,
    thick
]
coordinates {(\C,\YC)};

\addplot[
    only marks,
    mark=o,
    thick
]
coordinates {(\C,0)};

\draw[thick]
    (axis cs:\P,-0.08) --
    (axis cs:\P,0.08);
coordinates {(\P,0)};

\node[below left] at (axis cs:\C,0) {$c$};

\node[below right] at (axis cs:\P,0)
{$p$};

\node[left] at (axis cs:0,\D) {$\bar d$};

\node[
    left,
    xshift=-4pt
] at (axis cs:\C,\YC) {};



\end{axis}
\end{tikzpicture}
} 
\caption{The two cases of \eqref{eq:phi2lin} for $\lambda = 0$: left $p \leq c$, right $p > c$}
\label{fig:phi-c-min}
\end{figure}

\noindent
In both cases, we obtain a two-piece linear function that is not differentiable at $\min\{p,c\}$, and is globally concave if $p \leq c$. Note that, $\bar{z} \geq \min\{p,c\}$, otherwise the definition \eqref{eq:phi2lin} would contradict Proposition \ref{prop:zbar}.


\subsubsection{The case $\lambda > 0$}\label{sssec:philamg0}

Analogously to the previous case, we want to understand when $h(r) + {\bar d} < 0$, so that $\phi(z) < 0$. Since $r\geq \rho > 0$, $h(r)$ is strictly convex, differentiable, and diverges to $+\infty$ as $r \to 0^+$ and $r \to +\infty$, attaining a minimum at $\bar{v}$. Clearly, if $h(\bar{v}) \geq 0$ then $\phi(z) = 0$ for all values of $z$, as before: therefore, we only examine the case $h(\bar{v}) < 0$, which also implies that the function admits two roots $0 \leq \zeta_- \leq \zeta_+$. Hence, there are three cases:
\begin{description}
 \item[(i)] if  $\bar{v}\le c$, then\\[-0.7cm]
            \begin{equation}
		    \phi(z)=
		    \begin{cases}
			h({\bar v}) + {\bar d} \ \ &\text{if} \ \ z\le\bar{v}\\
			h(z) + {\bar d} \ \ &\text{if} \ \ \bar{v}\le z\le\min\{\zeta_+,c\}\\
			0 \ \ & \text{otherwise}
		    \end{cases}
		    \label{eq:LG0case1}
	        \end{equation}
 \item[(ii)] if $\zeta_- < c < \bar{v}$, then\\[-0.7cm]
             \begin{equation}
		      \phi(z)=
		      \begin{cases}
			   h(c) + {\bar d} \ \ &\text{if} \ \ z\le c\\
			   0 \ \ & \text{otherwise}
		      \end{cases}
 		    \label{eq:LG0case2}
	         \end{equation}
 \item[(iii)] If $c \leq \zeta_-$, then $\phi(z)= 0$ for all $z$.
\end{description}
In case {\bf (i)}, the function is two-piece convex. It starts from a constant negative value, and its derivative is equal to zero until the point $\bar v$ is reached. After this point, the dependence on $z$ becomes explicit, and the derivative of the function becomes positive. At $z =\bar v$, both the function and its derivative are continuous. Although the function changes its definition at this point, convexity is preserved because the change occurs exactly at the minimum of the nonlinear component. The function then changes definition again at $\min\{\zeta_+,c\}$. If this point coincides with $\zeta_+$, the function remains continuous, as shown in Figure \ref{fig:phi-nl-i2}(left). However, its derivative drops to zero after being positive on the previous interval, and therefore the function is no longer (globally) convex. If, on the other hand, $c \leq \zeta_+$, both continuity and convexity fail, as shown in Figure \ref{fig:phi-nl-i2}(centre). The behaviour is similar to the case $\lambda=0$, where the function increases up to the point $c$, then jumps to zero, and its derivative also drops to zero. In case {\bf (ii)}, the function again starts from a constant negative value and jumps to zero at $c$. Hence, neither continuity nor convexity is preserved, yet the function is two-piece linear, as shown in Figure \ref{fig:phi-nl-i2}(right). In case {\bf (iii)}, the function is equal to zero everywhere. Note that $\bar{z} \geq \min\{\zeta_+,c\}$ due to Proposition \ref{prop:zbar}.

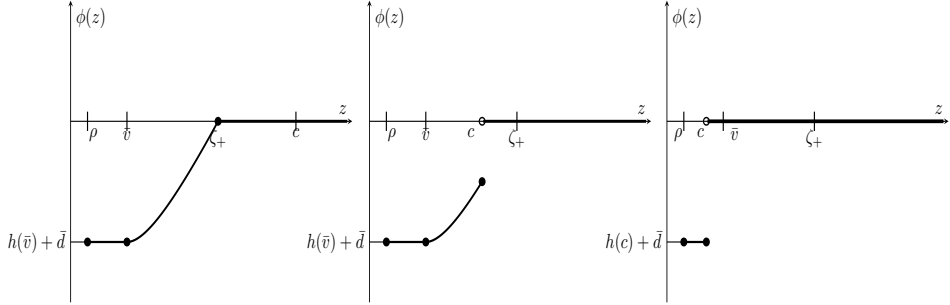
\begin{figure}[htbp]
\resizebox{\textwidth}{4cm}{
\begin{tikzpicture}
\begin{axis}[
    axis lines=middle,
    axis line style={thin},
    xlabel={$z$},
    ylabel={$\phi(z)$},
    xmin=0, xmax=5,
    ymin=-1.5, ymax=1,
    xtick=\empty,
    ytick=\empty,
    clip=false
    ]

\pgfmathsetmacro{\F}{1}
\pgfmathsetmacro{\LamL}{1}
\pgfmathsetmacro{\LamLbar}{0}
\pgfmathsetmacro{\dbar}{-3}

\pgfmathsetmacro{\Rho}{0.3}
\pgfmathsetmacro{\vbar}{sqrt(\LamL/\F)}
\pgfmathsetmacro{\Zeta}{(3+sqrt(5))/2}  
\pgfmathsetmacro{\C}{4}

\pgfmathsetmacro{\Hv}{\F*\vbar + \LamL/\vbar + \LamLbar + \dbar}
\pgfmathsetmacro{\Hzeta}{\F*\Zeta + \LamL/\Zeta + \LamLbar + \dbar}

\addplot[
    thin,
    domain=0:\Rho,
    samples=100
]
{\Hv};

\addplot[
    very thick,
    domain=\Rho:\vbar,
    samples=2
]
{\Hv};

\addplot[
    very thick,
    domain=\vbar:\Zeta,
    samples=200
]
{\F*x + \LamL/x + \LamLbar + \dbar};

\draw[
    line width=1.6pt
]
(axis cs:\Zeta,0) -- (axis cs:4.91,0);

\draw[thick]
    (axis cs:\Rho,-0.08) --
    (axis cs:\Rho,0.08);
coordinates {(\Rho,0)};

coordinates {(\Rho,0)};
\node[below] at (axis cs:\Rho+0.1,0) {$\rho$};

\addplot[
    only marks,
    mark=*,
    thick
]
coordinates {(\Rho,\Hv)};

\addplot[
    only marks,
    mark=*,
    thick
]
coordinates {(\vbar,\Hv)};
\node[below] at (axis cs:\vbar,0) {$\bar v$};

\addplot[
    only marks,
    mark=*,
    thick
]
coordinates {(\Zeta,0)};
\node[below] at (axis cs:\Zeta,0) {$\zeta_+$};

\draw[thick]
    (axis cs:\C,-0.08) --
    (axis cs:\C,0.08);
coordinates {(\C,0)};
\node[below] at (axis cs:\C,0) {$c$};

\node[left] at (axis cs:0,\Hv) {$h(\bar v)+\bar d$};

\node[above right] at (axis cs:1.55,-0.3) {};


\draw[thick]
    (axis cs:\vbar,-0.08) --
    (axis cs:\vbar,0.08);
coordinates {(\vbar,0)};

\end{axis}
\end{tikzpicture}
\hspace{-1.5cm}
\begin{tikzpicture}
\begin{axis}[
    axis lines=middle,
    axis line style={thin},
    xlabel={$z$},
    ylabel={$\phi(z)$},
    xmin=0, xmax=5,
    ymin=-1.5, ymax=1,
    xtick=\empty,
    ytick=\empty,
    clip=false
]

\pgfmathsetmacro{\F}{1}
\pgfmathsetmacro{\LamL}{1}
\pgfmathsetmacro{\LamLbar}{0}
\pgfmathsetmacro{\dbar}{-3}

\pgfmathsetmacro{\Rho}{0.3}
\pgfmathsetmacro{\vbar}{sqrt(\LamL/\F)}
\pgfmathsetmacro{\C}{2}
\pgfmathsetmacro{\Zeta}{(3+sqrt(5))/2}  

\pgfmathsetmacro{\Hv}{\F*\vbar + \LamL/\vbar + \LamLbar + \dbar}
\pgfmathsetmacro{\HC}{\F*\C + \LamL/\C + \LamLbar + \dbar}
\pgfmathsetmacro{\Hzeta}{\F*\Zeta + \LamL/\Zeta + \LamLbar + \dbar}

\addplot[
    thin,
    domain=0:\Rho,
    samples=100
]
{\Hv};

\addplot[
    very thick,
    domain=\Rho:\vbar,
    samples=2
]
{\Hv};

\addplot[
    very thick,
    domain=\vbar:\C,
    samples=200
]
{\F*x + \LamL/x + \LamLbar + \dbar};

\draw[
    line width=1.6pt
]
(axis cs:\C+0.02,0) -- (axis cs:4.91,0);


\draw[thick]
    (axis cs:\Rho,-0.08) --
    (axis cs:\Rho,0.08);
coordinates {(\Rho,0)};
\node[below] at (axis cs:\Rho+0.1,0) {$\rho$};

\addplot[
    only marks,
    mark=*,
    thick
]
coordinates {(\Rho,\Hv)};

\addplot[
    only marks,
    mark=*,
    thick
]
coordinates {(\vbar,\Hv)};


\draw[thick]
    (axis cs:\vbar,-0.08) --
    (axis cs:\vbar,0.08);
coordinates {(\vbar,0)};
\node[below] at (axis cs:\vbar,0) {$\bar v$};

\addplot[
    only marks,
    mark=*,
    thick
]
coordinates {(\C,\HC)};

\addplot[
    only marks,
    mark=o,
    thick
]
coordinates {(\C,0)};

\node[below] at (axis cs:\C-0.2,0) {$c$};

\draw[thick]
    (axis cs:\Zeta,-0.08) --
    (axis cs:\Zeta,0.08);
coordinates {(\Zeta,0)};
\node[below] at (axis cs:\Zeta,0) {$\zeta_+$};

\node[left, xshift=-4pt] at (axis cs:\C,\HC) {};

\node[left] at (axis cs:0,\Hv) {$h(\bar v)+\bar d$};



\end{axis}
\end{tikzpicture}
\hspace{-1.5cm}
\begin{tikzpicture}
\begin{axis}[
    axis lines=middle,
    axis line style={thin},
    xlabel={$z$},
    ylabel={$\phi(z)$},
    xmin=0, xmax=5,
    ymin=-1.5, ymax=1,
    xtick=\empty,
    ytick=\empty,
    clip=false
]

\pgfmathsetmacro{\F}{1}
\pgfmathsetmacro{\LamL}{1}
\pgfmathsetmacro{\LamLbar}{0}
\pgfmathsetmacro{\dbar}{-3}

\pgfmathsetmacro{\Rho}{0.3}
\pgfmathsetmacro{\C}{0.7}
\pgfmathsetmacro{\vbar}{sqrt(\LamL/\F)}
\pgfmathsetmacro{\Zeta}{(3+sqrt(5))/2}  

\pgfmathsetmacro{\Hv}{\F*\vbar + \LamL/\vbar + \LamLbar + \dbar}

\addplot[
    thin,
    domain=0:\Rho,
    samples=2
]
{\Hv};

\addplot[
    very thick,
    domain=\Rho:\C,
    samples=2
]
{\Hv};

\draw[
    line width=2.0pt
]
(axis cs:\C+0.02,0) -- (axis cs:4.91,0);



\draw[thick]
    (axis cs:\Rho,-0.08) --
    (axis cs:\Rho,0.08);
coordinates {(\Rho,0)};
\node[below] at (axis cs:\Rho-0.1,0) {$\rho$};

\addplot[
    only marks,
    mark=*,
    thick
]
coordinates {(\Rho,\Hv)};

\addplot[
    only marks,
    mark=*,
    thick
]
coordinates {(\C,\Hv)};

\addplot[
    only marks,
    mark=o,
    thick
]
coordinates {(\C,0)};

\node[below] at (axis cs:\C-0.1,0) {$c$};

\draw[thick]
    (axis cs:\vbar,-0.08) --
    (axis cs:\vbar,0.08);
coordinates {(\vbar,0)};
\node[below] at (axis cs:\vbar+0.2,0) {$\bar v$};

\draw[thick]
    (axis cs:\Zeta,-0.08) --
    (axis cs:\Zeta,0.08);
coordinates {(\Zeta,0)};
\node[below] at (axis cs:\Zeta,0) {$\zeta_+$};

\node[left] at (axis cs:0,\Hv) {$h(c)+\bar d$};



\end{axis}
\end{tikzpicture}
}
\caption{$\lambda > 0$: \eqref{eq:LG0case1} with $c \geq \zeta_+$ (left) and $\zeta_+ \geq c$ (center) and \eqref{eq:LG0case2} (right)}
\label{fig:phi-nl-i2}
\end{figure}

\subsection{Approximations of $\phi(z)$}\label{ssec:approxphi}

The analysis of $\phi(z)$ allows to construct a two-piece linear (lower) approximation for it. The left and right piece of our approximation correspond to the left and right interval of the domain with respect to the point $\bar{z}$, namely $V_1= [\rho,\, \bar{z}]$ and $V_2[\bar{z},\, c_{max}]$, respectively. According to Proposition \ref{prop:zbar}, $\phi(z) = 0, \forall\, z \in V_2$, so the right piece of the approximation is always zero (i.e., it coincides with the original function value). Therefore we focus on how to (lower) approximate the left piece with a linear function $\phi^-(z)$. In the following, the notation $\ell(a,b)_{\phi}$ is used to represent the line passing through the points $\left(a,\phi(a)\right)$ and $\left(b,\phi(b)\right)$.

\subsubsection{The case $\lambda = 0$}

Since $\bar{z} \geq \min\{p,c\}$, we construct a line that underestimates $\phi(z)$ in $V_1$ as follows:
%
%
\begin{equation}
		\phi^-(z)=\begin{cases}
			\ell(\rho,\bar{z})_{\phi} \ \ &\text{if} \ \ \bar{z}\ge p\\
			 \ell(c,\bar{z})_{\phi} \ \ & \text{otherwise}
		\end{cases}
\end{equation}
Clearly, the case $\min\{p,c\} = p$ implies $\bar{z} \geq p$, as shown in Figure \ref{fig:phiapp-LZ}(left). If, instead, $\min\{p,c\} = c$, we need to distinguish two cases, as shown in Figure \ref{fig:phiapp-LZ}(centre and right). Note that the relative position between $\bar{z}$ and $c$ is irrelevant.

\begin{figure}[htbp]
\resizebox{\textwidth}{4cm}{
\begin{tikzpicture}
\begin{axis}[
    axis lines=middle,
    axis line style={thin},
    xlabel={$z$},
    ylabel={$\phi(z)$},
    xmin=-1, xmax=6,
    ymin=-3, ymax=1,
    xtick=\empty,
    ytick=\empty,
    clip=false
]

\pgfmathsetmacro{\P}{3}       
\pgfmathsetmacro{\Rho}{0.4}   
\pgfmathsetmacro{\C}{4}       
\pgfmathsetmacro{\F}{1}       
\pgfmathsetmacro{\D}{-3}      

\pgfmathsetmacro{\Yr}{\F*\Rho+\D}
\pgfmathsetmacro{\Yp}{\F*\P+\D}

\pgfmathsetmacro{\zbar}{(\P+\C)/2}

\draw[thick]
    (axis cs:\zbar,-0.08) --
    (axis cs:\zbar,0.08);

\addplot[
    only marks,
    mark=*,
    thick
]
coordinates {(\zbar,0)};

\node[above] at (axis cs:\zbar,+0.1) {$\bar z$};

\addplot[
    thin,
    domain=0:\Rho,
    samples=100
]
{\F*x + \D};

\addplot[
    very thick,
    domain=\Rho:\P,
    samples=100
]
{\F*x + \D};


\draw[
    line width=1.6pt
]
(axis cs:\P,0) -- (axis cs:6,0);

\draw[thick]
    (axis cs:\Rho,-0.08) --
    (axis cs:\Rho,0.08);
coordinates {(\Rho,0)};

\addplot[
    only marks,
    mark=*,
    thick
]
coordinates {(\Rho,\Yr)};

\node[below] at (axis cs:\Rho+0.2,0) {$\rho$};

\addplot[
    only marks,
    mark=*,
    thick
]
coordinates {(\P,0)};

\node[above] at (axis cs:\P,0) {$p$};

\draw[thick]
    (axis cs:\C,-0.08) --
    (axis cs:\C,0.08);
coordinates {(\C,0)};

\node[above] at (axis cs:\C,0) {$c$};

\node[left] at (axis cs:0,\D) {$\bar d$};

\addplot[
    very thick,
    dashed
]
coordinates {(\Rho,\Yr) (\zbar,0)};

\end{axis}
\end{tikzpicture}
\begin{tikzpicture}
\begin{axis}[
    axis lines=middle,
    xlabel={$z$},
    ylabel={$\phi(z)$},
    xmin=-1, xmax=6,
    ymin=-5, ymax=1,
    xtick=\empty,
    ytick=\empty,
    clip=false
]

\pgfmathsetmacro{\F}{1}
\pgfmathsetmacro{\D}{-3}
\pgfmathsetmacro{\P}{-\D/\F}   
\pgfmathsetmacro{\C}{2}        
\pgfmathsetmacro{\Rho}{0.4}    

\pgfmathsetmacro{\YC}{\F*\C+\D}
\pgfmathsetmacro{\Yr}{\F*\Rho+\D}

\pgfmathsetmacro{\zbar}{(\C+\P)/2}

\pgfmathsetmacro{\mapp}{-\YC/(\zbar-\C)}

\pgfmathsetmacro{\YappRho}{\mapp*(\Rho-\zbar)}

\addplot[
    thin,
    domain=0:\Rho,
    samples=100
]
{\F*x + \D};

\addplot[
    very thick,
    domain=\Rho:\C,
    samples=100
]
{\F*x + \D};

\addplot[
    very thick,
    domain=\C:6,
    samples=2
]
{0};


\draw[thick]
    (axis cs:\Rho,-0.08) --
    (axis cs:\Rho,0.08);
coordinates {(\Rho,0)};

\addplot[
    only marks,
    mark=*,
    thick
]
coordinates {(\Rho,\Yr)};

\node[below] at (axis cs:\Rho+0.2,0) {$\rho$};

\addplot[
    only marks,
    mark=*,
    thick
]
coordinates {(\C,\YC)};

\addplot[
    only marks,
    mark=o,
    thick
]
coordinates {(\C,0)};

\draw[thick]
    (axis cs:\P,-0.08) --
    (axis cs:\P,0.08);
coordinates {(\P,0)};

\node[below left] at (axis cs:\C,0) {$c$};

\node[below right] at (axis cs:\P,0)
{$p$};

\node[left] at (axis cs:0,\D) {$\bar d$};

\node[
    left,
    xshift=-4pt
] at (axis cs:\C,\YC) {};



\addplot[
    only marks,
    mark=*,
    thick
]
coordinates {(\zbar,0)};

\node[above] at (axis cs:\zbar,0) {$\bar z$};

\addplot[
    very thick,
    dashed,
    domain=\Rho:\zbar,
    samples=2
]
{\mapp*(x-\zbar)};

\addplot[
    only marks,
    mark=*,
    thick
]
coordinates {(\Rho,\YappRho)};

\end{axis}
\end{tikzpicture}
\begin{tikzpicture}
\begin{axis}[
    axis lines=middle,
    xlabel={$z$},
    ylabel={$\phi(z)$},
    xmin=-1, xmax=6,
    ymin=-3, ymax=1,
    xtick=\empty,
    ytick=\empty,
    clip=false
]

\pgfmathsetmacro{\F}{1}
\pgfmathsetmacro{\D}{-3}
\pgfmathsetmacro{\P}{-\D/\F}        
\pgfmathsetmacro{\C}{2}             
\pgfmathsetmacro{\Rho}{0.4}         
\pgfmathsetmacro{\zbar}{(\P+6)/2}   

\pgfmathsetmacro{\YC}{\F*\C+\D}
\pgfmathsetmacro{\Yr}{\F*\Rho+\D}

\addplot[
    thin,
    domain=0:\Rho,
    samples=100
]
{\F*x + \D};

\addplot[
    very thick,
    domain=\Rho:\C,
    samples=100
]
{\F*x + \D};

\addplot[
    very thick,
    domain=\C:6,
    samples=2
]
{0};

\draw[thick]
    (axis cs:\Rho,-0.08) --
    (axis cs:\Rho,0.08);

\addplot[
    only marks,
    mark=*,
    thick
]
coordinates {(\Rho,\Yr)};

\node[below] at (axis cs:\Rho+0.2,0) {$\rho$};

\addplot[
    only marks,
    mark=*,
    thick
]
coordinates {(\C,\YC)};

\addplot[
    only marks,
    mark=o,
    thick
]
coordinates {(\C,0)};

\draw[thick]
    (axis cs:\P,-0.08) --
    (axis cs:\P,0.08);

\node[below left] at (axis cs:\C,0) {$c$};

\node[below right] at (axis cs:\P,0) {$p$};

\draw[thick]
    (axis cs:\zbar,-0.08) --
    (axis cs:\zbar,0.08);

\node[below] at (axis cs:\zbar,0) {$\bar z$};

\addplot[
    very thick,
    dashed
]
coordinates {(\Rho,\Yr) (\zbar,0)};

\node[left] at (axis cs:0,\D) {$\bar d$};

\addplot[
    only marks,
    mark=*,
    thick
]
coordinates {(\zbar,0)};

\end{axis}
\end{tikzpicture}
}
\caption{Approximating the case $\lambda = 0$: $p \leq c \implies \bar{z} \ge p$ (left), $c \leq p$ and $\bar{z} \geq p$ (centre), $c \leq p$ and $\bar{z} \leq p$ (right)}
\label{fig:phiapp-LZ}
\end{figure}
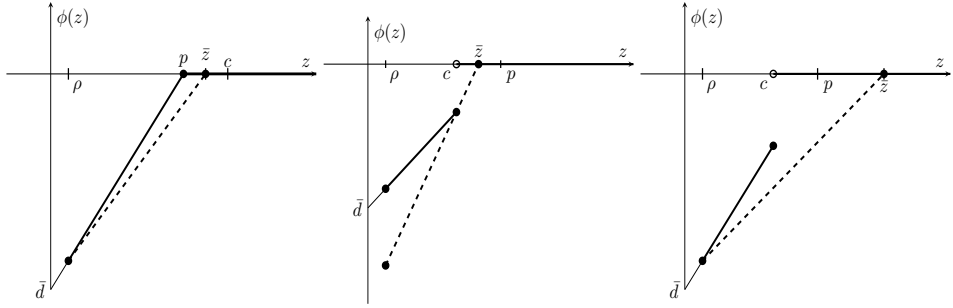

\subsubsection{The case $\lambda > 0$}

We consider the same three cases described in \S \ref{sssec:philamg0}.
\begin{itemize}
\item[(i)] If $\bar{v}\le c$, the function $\phi(z)$ consists of a nonlinear convex part and a zero part, possibly with a jump at $c$. We first consider the line joining the points on $\phi(z)$ corresponding to $\bar{v}$ and $\bar{z}$. In general, this line intersects $\phi(z)$; therefore, to obtain a valid lower approximation, we shift it downward until it becomes tangent to $h(z)$, i.e., we take $\phi^-(z)=\ell^{\rightarrow}(\bar{v},\bar{z})_{\phi}$ as shown in Figure \ref{fig:appphiLG0}(left and centre).
\item[(ii)] If $\zeta_- < c < \bar{v}$, the function $\phi(z)$ consists of two linear constant pieces with a jump at $c$. To construct a linear lower approximation, we simply take the line joining the points on $\phi(z)$ corresponding to $c$ and $\bar{z}$, i.e., $\phi^-(z)=\ell(c,\bar{z})_{\phi}$ as shown in Figure \ref{fig:appphiLG0}(right).
\item[(iii)] If $c \leq \zeta_-$, the function $\phi(z)$ is identically zero, and therefore it coincides with its approximation: $\phi^-(z)= \phi(z)$.
\end{itemize}

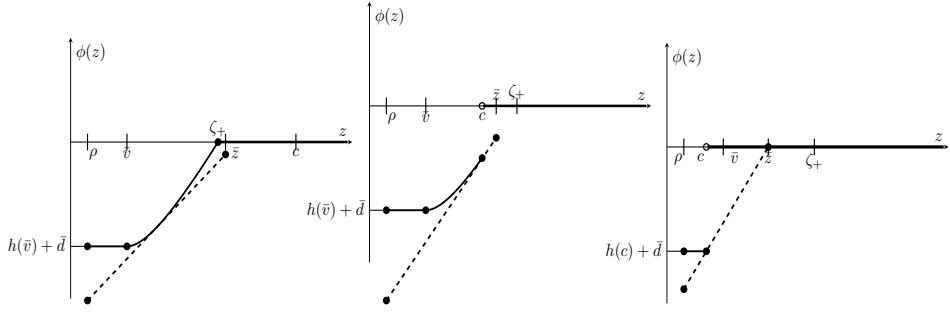
\begin{figure}[htbp]
\resizebox{\textwidth}{4cm}{
\begin{tikzpicture}
\begin{axis}[
    axis lines=middle,
    axis line style={thin},
    xlabel={$z$},
    ylabel={$\phi(z)$},
    xmin=0, xmax=5,
    ymin=-1.5, ymax=1,
    xtick=\empty,
    ytick=\empty,
    clip=false
]

\pgfmathsetmacro{\F}{1}
\pgfmathsetmacro{\LamL}{1}
\pgfmathsetmacro{\LamLbar}{0}
\pgfmathsetmacro{\dbar}{-3}

\pgfmathsetmacro{\Rho}{0.3}
\pgfmathsetmacro{\vbar}{sqrt(\LamL/\F)}
\pgfmathsetmacro{\Zeta}{(3+sqrt(5))/2}
\pgfmathsetmacro{\zbar}{2.75}       
\pgfmathsetmacro{\C}{4}

\pgfmathsetmacro{\Hv}{\F*\vbar + \LamL/\vbar + \LamLbar + \dbar}
\pgfmathsetmacro{\Hzeta}{\F*\Zeta + \LamL/\Zeta + \LamLbar + \dbar}


\pgfmathsetmacro{\mapp}{-\Hv/(\zbar-\vbar)}

\pgfmathsetmacro{\ztan}{sqrt(\LamL/(\F-\mapp))}

\pgfmathsetmacro{\Gtan}{\F*\ztan + \LamL/\ztan + \LamLbar + \dbar}

\pgfmathsetmacro{\Yzbar}{\mapp*(\zbar-\ztan)+\Gtan}

\pgfmathsetmacro{\YrhoApp}{\mapp*(\Rho-\ztan)+\Gtan}


\addplot[
    thin,
    domain=0:\Rho,
    samples=100
]
{\Hv};

\addplot[
    very thick,
    domain=\Rho:\vbar,
    samples=2
]
{\Hv};

\addplot[
    very thick,
    domain=\vbar:\Zeta,
    samples=200
]
{\F*x + \LamL/x + \LamLbar + \dbar};

\draw[
    line width=1.6pt
]
(axis cs:\Zeta,0) -- (axis cs:4.91,0);

\draw[thick]
    (axis cs:\Rho,-0.08) --
    (axis cs:\Rho,0.08);

\node[below] at (axis cs:\Rho+0.1,0) {$\rho$};

\addplot[
    only marks,
    mark=*,
    thick
]
coordinates {(\Rho,\Hv)};

\addplot[
    only marks,
    mark=*,
    thick
]
coordinates {(\vbar,\Hv)};

\draw[thick]
    (axis cs:\vbar,-0.08) --
    (axis cs:\vbar,0.08);

\node[below] at (axis cs:\vbar,0) {$\bar v$};

\addplot[
    only marks,
    mark=*,
    thick
]
coordinates {(\Zeta,0)};

\node[above] at (axis cs:\Zeta,0) {$\zeta_+$};

\draw[thick]
    (axis cs:\zbar,-0.08) --
    (axis cs:\zbar,0.08);

\node[below right] at (axis cs:\zbar,0) {$\bar z$};

\draw[thick]
    (axis cs:\C,-0.08) --
    (axis cs:\C,0.08);

\node[below] at (axis cs:\C,0) {$c$};

\node[left] at (axis cs:0,\Hv) {$h(\bar v)+\bar d$};


\addplot[
    very thick,
    dashed,
    domain=\Rho:\zbar,
    samples=2
]
{\mapp*(x-\ztan)+\Gtan};



\addplot[
    only marks,
    mark=*,
    thick
]
coordinates {(\zbar,\Yzbar)};

\draw[dotted]
    (axis cs:\zbar,0) --
    (axis cs:\zbar,\Yzbar);

\addplot[
    only marks,
    mark=*,
    thick
]
coordinates {(\Rho,\YrhoApp)};

\end{axis}
\end{tikzpicture}
\hspace{-1.5cm}
\begin{tikzpicture}
\begin{axis}[
    axis lines=middle,
    axis line style={thin},
    xlabel={$z$},
    ylabel={$\phi(z)$},
    xmin=0, xmax=5,
    ymin=-1.5, ymax=1,
    xtick=\empty,
    ytick=\empty,
    clip=false
]

\pgfmathsetmacro{\F}{1}
\pgfmathsetmacro{\LamL}{1}
\pgfmathsetmacro{\LamLbar}{0}
\pgfmathsetmacro{\dbar}{-3}

\pgfmathsetmacro{\Rho}{0.3}
\pgfmathsetmacro{\vbar}{sqrt(\LamL/\F)}
\pgfmathsetmacro{\C}{2}
\pgfmathsetmacro{\zbar}{2.25}        
\pgfmathsetmacro{\Zeta}{(3+sqrt(5))/2}

\pgfmathsetmacro{\Hv}{\F*\vbar + \LamL/\vbar + \LamLbar + \dbar}
\pgfmathsetmacro{\HC}{\F*\C + \LamL/\C + \LamLbar + \dbar}
\pgfmathsetmacro{\Hzeta}{\F*\Zeta + \LamL/\Zeta + \LamLbar + \dbar}


\pgfmathsetmacro{\mapp}{-\Hv/(\zbar-\vbar)}

\pgfmathsetmacro{\ztan}{sqrt(\LamL/(\F-\mapp))}

\pgfmathsetmacro{\Gtan}{\F*\ztan + \LamL/\ztan + \LamLbar + \dbar}

\pgfmathsetmacro{\Yzbar}{\mapp*(\zbar-\ztan)+\Gtan}

\pgfmathsetmacro{\YrhoApp}{\mapp*(\Rho-\ztan)+\Gtan}


\addplot[
    thin,
    domain=0:\Rho,
    samples=100
]
{\Hv};

\addplot[
    very thick,
    domain=\Rho:\vbar,
    samples=2
]
{\Hv};

\addplot[
    very thick,
    domain=\vbar:\C,
    samples=200
]
{\F*x + \LamL/x + \LamLbar + \dbar};

\draw[
    line width=1.6pt
]
(axis cs:\C+0.02,0) -- (axis cs:4.91,0);

\draw[thick]
    (axis cs:\Rho,-0.08) --
    (axis cs:\Rho,0.08);

\node[below] at (axis cs:\Rho+0.1,0) {$\rho$};

\addplot[
    only marks,
    mark=*,
    thick
]
coordinates {(\Rho,\Hv)};

\addplot[
    only marks,
    mark=*,
    thick
]
coordinates {(\vbar,\Hv)};

\draw[thick]
    (axis cs:\vbar,-0.08) --
    (axis cs:\vbar,0.08);

\node[below] at (axis cs:\vbar,0) {$\bar v$};

\addplot[
    only marks,
    mark=*,
    thick
]
coordinates {(\C,\HC)};

\addplot[
    only marks,
    mark=o,
    thick
]
coordinates {(\C,0)};

\node[below] at (axis cs:\C,0) {$c$};

\draw[thick]
    (axis cs:\zbar,-0.08) --
    (axis cs:\zbar,0.08);

\node[above] at (axis cs:\zbar,0) {$\bar z$};

\draw[thick]
    (axis cs:\Zeta,-0.08) --
    (axis cs:\Zeta,0.08);

\node[above] at (axis cs:\Zeta,0) {$\zeta_+$};

\node[left] at (axis cs:0,\Hv) {$h(\bar v)+\bar d$};


\addplot[
    very thick,
    dashed,
    domain=\Rho:\zbar,
    samples=2
]
{\mapp*(x-\ztan)+\Gtan};



\addplot[
    only marks,
    mark=*,
    thick
]
coordinates {(\zbar,\Yzbar)};


\addplot[
    only marks,
    mark=*,
    thick
]
coordinates {(\Rho,\YrhoApp)};

\end{axis}
\end{tikzpicture}
\hspace{-1.5cm}
\begin{tikzpicture}
\begin{axis}[
    axis lines=middle,
    axis line style={thin},
    xlabel={$z$},
    ylabel={$\phi(z)$},
    xmin=0, xmax=5,
    ymin=-1.5, ymax=1,
    xtick=\empty,
    ytick=\empty,
    clip=false
]

\pgfmathsetmacro{\F}{1}
\pgfmathsetmacro{\LamL}{1}
\pgfmathsetmacro{\LamLbar}{0}
\pgfmathsetmacro{\dbar}{-3}

\pgfmathsetmacro{\Rho}{0.3}
\pgfmathsetmacro{\C}{0.7}
\pgfmathsetmacro{\vbar}{sqrt(\LamL/\F)}
\pgfmathsetmacro{\zbar}{1.8}          
\pgfmathsetmacro{\Zeta}{(3+sqrt(5))/2}

\pgfmathsetmacro{\Hv}{\F*\vbar + \LamL/\vbar + \LamLbar + \dbar}

\pgfmathsetmacro{\YC}{\Hv}

\pgfmathsetmacro{\mapp}{-\YC/(\zbar-\C)}

\pgfmathsetmacro{\YappRho}{\mapp*(\Rho-\zbar)}

\addplot[
    thin,
    domain=0:\Rho,
    samples=2
]
{\Hv};

\addplot[
    very thick,
    domain=\Rho:\C,
    samples=2
]
{\Hv};

\draw[
    line width=2.0pt
]
(axis cs:\C+0.02,0) -- (axis cs:4.91,0);

\draw[thick]
    (axis cs:\Rho,-0.08) --
    (axis cs:\Rho,0.08);

\node[below] at (axis cs:\Rho-0.1,0) {$\rho$};

\addplot[
    only marks,
    mark=*,
    thick
]
coordinates {(\Rho,\Hv)};

\addplot[
    only marks,
    mark=*,
    thick
]
coordinates {(\C,\Hv)};

\addplot[
    only marks,
    mark=o,
    thick
]
coordinates {(\C,0)};

\node[below] at (axis cs:\C-0.1,0) {$c$};

\draw[thick]
    (axis cs:\vbar,-0.08) --
    (axis cs:\vbar,0.08);

\node[below] at (axis cs:\vbar+0.2,0) {$\bar v$};

\draw[thick]
    (axis cs:\zbar,-0.08) --
    (axis cs:\zbar,0.08);

\node[below] at (axis cs:\zbar,0) {$\bar z$};

\draw[thick]
    (axis cs:\Zeta,-0.08) --
    (axis cs:\Zeta,0.08);

\node[below] at (axis cs:\Zeta,0) {$\zeta_+$};

\node[left] at (axis cs:0,\Hv) {$h(c)+\bar d$};


\addplot[
    very thick,
    dashed,
    domain=\Rho:\zbar,
    samples=2
]
{\mapp*(x-\zbar)};

\addplot[
    only marks,
    mark=*,
    thick
]
coordinates {(\Rho,\YappRho)};

\addplot[
    only marks,
    mark=*,
    thick
]
coordinates {(\zbar,0)};

\end{axis}
\end{tikzpicture}
}
\caption{Approximations for $\lambda > 0$: \eqref{eq:LG0case1} with $c \geq \zeta_+$ (left) and $\zeta_+ \geq c$ (center) and \eqref{eq:LG0case2} (right)}
\label{fig:appphiLG0}
\end{figure}

\subsection{Overall approximation of $\psi(z)$}

The previous analysis finally provides a lower approximation of the value function $\underline{v}(z)$. With, as usual, $\lambda$ and $\pi$ those computed in the given $\bar z$, the original function is
\[
 \textstyle
 \psi(z)= -\lambda \delta - \sum_{i \in N} \pi_i b_i
 + \frac{\lambda \sigma}{z} + \sum_{(i,j) \in A}\phi_{ij}(z)
 \; .
\]
The term $-\lambda \delta - \sum_{i \in N} \pi_i b_i$ is a constant and therefore we just carry it through in the approximation. The term $g(z) = \lambda \sigma / z$ is convex and differentiable, and therefore its first-order model $g(\bar{z}) + g'(\bar{z})( z - \bar{z} )$ is a valid lower approximation everywhere, with $g'(\bar{z}) = - \lambda\sigma / \bar{z}^2$. Then, for each arc $(i,j) \in A$, we replace the corresponding function $\phi_{ij}(z)$ with its two-piece linear approximation described in \S \ref{ssec:approxphi}. In particular, since each two-piece linear function has the same breakpoint at $\bar{z}$, their sum is still a two-piece linear function with a breakpoint at $\bar{z}$. Hence, we obtain an overall two-piecewise-linear nonconvex approximation $\ell_-(z)$ on $z \in [ \, \rho \,,\, \bar{z} \, ]$ and $\ell_+(z)$ on $z \in [ \, \bar{z} \,,\, c \, ]$, as illustrated in Figure \ref{fig:appall}. This allows us to solve the problem of minimizing $\underline{v}(z)$ on $[ \, \rho \,,\, c \, ]$ with a \emph{spatial Branch-and-Bound} approach, as discussed next.

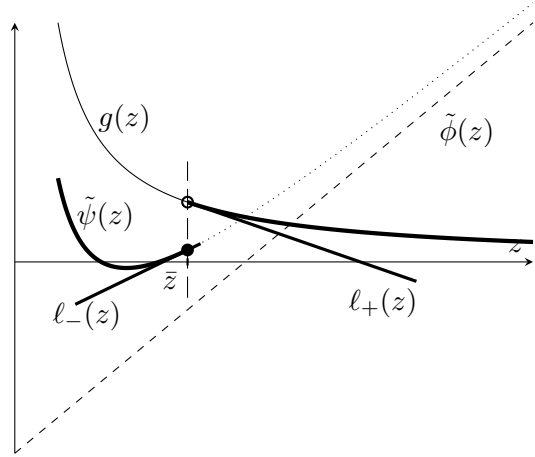
\begin{figure}[htbp]
\centering
\begin{tikzpicture}
\begin{axis}[
    axis lines=middle,
    axis line style={thin},
    xlabel={$z$},
    ylabel={},
    xmin=0, xmax=6,
    ymin=-8, ymax=10,
    xtick=\empty,
    ytick=\empty,
    clip=false,
    samples=200
]

\pgfmathsetmacro{\zbar}{2}

\pgfmathsetmacro{\Hzbar}{5/\zbar + 3*\zbar - 8}
\pgfmathsetmacro{\Fzbar}{5/\zbar}

\pgfmathsetmacro{\mleft}{-5/(\zbar*\zbar) + 3}
\pgfmathsetmacro{\mright}{-5/(\zbar*\zbar)}

\addplot[
    thin,
    domain=0.5:6
]
{5/x};

\addplot[
    thin,
    dashed,
    domain=0:6
]
{3*x - 8};

\addplot[
    thin,
    dotted,
    domain=0.5:6
]
{5/x + 3*x - 8};

\addplot[
    ultra thick,
    domain=0.5:\zbar,
    samples=200
]
{5/x + 3*x - 8};

\addplot[
    ultra thick,
    domain=\zbar:6,
    samples=200
]
{5/x};

\draw[dash pattern=on 8pt off 3pt]
    (axis cs:\zbar,-1.5) --
    (axis cs:\zbar,4.2);

\draw[thick]
    (axis cs:\zbar,-0.15) --
    (axis cs:\zbar,0.15);

\node[below left] at (axis cs:\zbar,0) {$\bar z$};

\addplot[
    only marks,
    mark=*,
    thick
]
coordinates {(\zbar,\Hzbar)};

\addplot[
    only marks,
    mark=o,
    thick
]
coordinates {(\zbar,\Fzbar)};

\addplot[
    very thick,
    domain=0.70:2.15,
    samples=2
]
{\Hzbar + \mleft*(x-\zbar)};

\addplot[
    very thick,
    domain=2.0:4.65,
    samples=2
]
{\Fzbar + \mright*(x-\zbar)};

\node[right] at (axis cs:0.85,6.0) {$g(z)$};

\node[right] at (axis cs:4.8,5.5) {$\tilde{\phi}(z)$};

\node[right] at (axis cs:0.6,2.0) {$\tilde{\psi}(z)$};

\node[below left] at (axis cs:1.35,-1.20) {$ \ell_-(z)$};

\node[below right] at (axis cs:3.75,-0.65) {$ \ell_+(z)$};

\end{axis}
\end{tikzpicture}

\caption{Overall two-piece linear approximation.}
\label{fig:appall}

\end{figure}


\section{The Full Algorithm}\label{sec:algo}

In this section we describe our full {\em spatial Benders} algorithm, which combines Benders decomposition with a spatial partitioning of the domain of the variable $z$. At each iteration, we build a piecewise-linear-lower approximation of the Benders function $\psi(z)$, starting from a fixed breakpoint $\bar z$. This breakpoint divides the domain into two subintervals. Each affine piece of the approximation is associated with the subinterval on which it is valid and where it acts as a lower approximation for the original Benders function, as described in Section \ref{sec:approx}. Then, the main step of the algorithm consists in performing a line search on each subinterval separately, to minimize the pointwise maximum of the affine cuts that are valid there. This gives one candidate point per subinterval, at which the Benders function is evaluated and a new piecewise-linear-lower approximation is constructed. The new affine cuts are added to the corresponding subintervals, and the procedure is repeated. The method can be interpreted as a {\em spatial Benders cutting-plane} scheme, since the domain is partitioned into regions, and on each region the Benders function is approximated from below by the maximum of the cuts generated so far. In the following we describe the algorithm in detail.

\subsection{Spatial Benders procedure}
\label{subsec:spatial-benders}

We keep the current set $\mathcal I$ of \emph{active intervals}, each a sub-interval of $V^0 = [\rho, \, c_{\max}]$, and initially $\mathcal I = \{ \, V^0 \, \}$. Each active interval $V \in \mathcal I$ is associated with a set $\mathcal C(V)$ containing all affine cuts valid on $V$, and therefore allows to define the {\em local} Benders' {\em restricted master problem} (RMP)
\begin{align*}
 LB_V
 =
 \min\quad & z_0
 \\
 & z_0\geq \ell(z)
 && \forall \ell\in\mathcal C(V),
 \\
 & z \in V
\end{align*}
minimizing their pointwise maximum. With $\bar z_V$ its optimal solution, $LB_V$ is a valid lower bound on the optimal value of the original problem restricted to the interval $V$. At each iteration of the algorithm, the \emph{global} RMP
\begin{equation}
 LB =
 \min_{} \{ \, LB_V \;:\; V \in {\mathcal I} \, \}
 \label{eq:global-lower-bound}
\end{equation}
therefore provides a \emph{globally valid} lower bound $LB$ associated with a global optimal solution $\bar z$. The corresponding Lagrangian dual problem \eqref{eq:LagDfix-minlp-alt} is solved, yielding the optimal multipliers $(\bar\pi,\bar\lambda)$. As discussed in Section  \ref{sec:approx}, the function $\psi(\bar\pi,\bar\lambda;z)$ is non-convex and cannot be directly included in the master problem through a single globally valid affine cut. Therefore, the domain of $z$ is treated {\em spatially}. In particular, the sub-interval $\bar{V} = [ \, \bar{a}_- \,,\, \bar{a}_+ \, ]$ to which $\bar{z}$ belongs is partitioned at $\bar z$ into the two sub-intervals
\begin{equation*}
 \bar{V}_- =  [ \, \bar{a}_- \,,\, \bar{z} \, ],
 \qquad
 \bar{V}_+ = [ \, \bar{z} \,,\, \bar{a}_+ \, ]
 \label{eq:initial-spatial-partition}
\end{equation*}
and, for each of those, a different affine lower cut is constructed. Denoting the two cuts by $\ell_-(z)$ and $\ell_+(z)$, they satisfy
\begin{align*}
 \ell_-(z)
 &\leq
 \psi(\bar\pi,\bar\lambda;z)
 && \forall z\in \bar{V}_-,
 \\
 \ell_+(z)
 &\leq
 \psi(\bar\pi,\bar\lambda;z)
 && \forall z\in \bar{V}_+.
\end{align*}
Thus, in general, each partitioning operation generates two child intervals and one new affine lower cut for each child. Note that each cut is {\em local}, since it is guaranteed to be valid only on the associated interval. Yet, every cut valid on the parent interval remains valid on each of its subintervals, i.e., the children inherit all the cuts valid on the parent and receive their corresponding new cut:
\begin{align*}
 \mathcal C(\bar{V}_-)
 &=
 \mathcal C(\bar{V})\cup\{\ell_-(z)\}
 \\
 \mathcal C(\bar{V}_+)
 &=
 \mathcal C(\bar{V})\cup\{\ell_+(z)\}
\end{align*}
and the parent interval is replaced by its two children
\begin{equation*}
 \mathcal I =
 \mathcal I \setminus \{ \, \bar{V} \, \} \cup
 \{ \, \bar{V}_- \,,\, \bar{V}_+ \, \}.
\end{equation*}
Each newly generated child is subsequently processed in the same way: the local RMP problem provides a new candidate point, which is used both to evaluate the Lagrangian dual problem and to divide the corresponding interval into two further children.

\smallskip
\noindent
As discussed in Section \ref{sec:lagsolve}, the solution of the Lagrangian dual also produces a valid upper bound, and the best of these found across the search provides, as usual, the \emph{incumbent} solution and its associated upper bound $UB$. Hence, an active node $V$ can be safely fathomed whenever $LB_V \geq UB -\varepsilon$, where $\varepsilon$ denotes the required optimality tolerance. At each iteration, therefore, the optimal value is bracketed between $UB$ and the the globally valid lower bound $LB$ of \eqref{eq:global-lower-bound}; the typical selection strategy of the next interval $\bar{V}$ to process being the one yielding $LB$, i.e., the classical \emph{best-first} strategy aiming at improving the global lower bound. The entire algorithm terminates when the active set $\mathcal I$ is empty, which by standard arguments proves that the global optimality gap satisfies $UB-LB\leq\varepsilon$.

\smallskip
\noindent
The proposed algorithm was implemented and tested. In the next section,
its computational behaviour is compared with that of a general-purpose
solver.


\section{Computational Experiments}\label{sec:results}

The {\em Spatial Benders} (SB) algorithm presented in the previous sections was coded in {\tt C++} in the context of the {\tt SMS++} project \cite{SMS++}. In particular, the specific repository
\begin{center}
 \url{https://gitlab.com/smspp/SingleFlowDCRBlock}
\end{center}
was released, containing both {\tt SingleFlowDCRBlock}, the {\tt SMS++ Block} describing the problem, and {\tt SingleFlowDCRBendersSolver} implementing SB within the general {\tt SMS++ Solver} interface. Owing to the standard approach in {\tt SMS++}, {\tt SingleFlowDCRBlock} implements both the ``physical representation'' of the DCR problem---the data of the problem, that specialised {\tt Solver} uses--and the ``abstract representation'', i.e., the {\tt Variable}, {\tt Constraint} and {\tt Objective} that implement the convex MINLP formulation \eqref{eq:of}--\eqref{eq:sfsprmin1}, so that it can be solved by off-the-shelf general-purpose solvers (in our experiments, {\tt Gurobi 13.0.0}). In particular, the crucial nonlinear constraint \eqref{eq:cdelay} was implemented using both the standard reformulation of the perspective term as a Second-Order Cone \cite{FrGe09a}, and the dynamic separation of Perspective Cuts \cite{FG2006}; the latter form was most often shown to be preferable, but to be fair we always reported the best result obtained between the two.

\smallskip
\noindent
All code was compiled with {\tt clang++ 18.1.3} and {\tt -O3} optimization level. All experiments were executed on a single thread of an AMD EPYC 9654 96-Core Processor running at 2.40 GHz, 1.5 TB of RAM, and operating under Ubuntu  24.04.4 LTS. For all approaches, a (relative) tolerance parameter of {\tt 1e-6} was  set.

\subsection{Test instances}

For the test instances we used the same generation procedure described in \cite{FGS2015a}, which we briefly review for the sake of completeness. We considered both real-world and synthetic network topologies. The real-world instances belong to three different classes: GARR \cite{garr}, Internet Topology Zoo \cite{topo} and SNDlib \cite{sndlib}. The synthetic class is generated using the {Waxman} model \cite{waxman}, with $n$ ranging from $100$ to $800$. The information on the topology is then combined with realistic telecommunication network parameters as well as traffic matrices (i.e., flows), that are generated using the {\tt FNSS} tool \cite{fnss}. We set $L = 1500$ bytes as the maximum size of any packet at all links. Link and node delays are all set to $L / w_{ij}$. Link capacities are chosen among $\{ 1, 10, 40 \}$ Gbps, according to the link's \emph{edge betweeness} \cite{fnss}. With regards to the flows, all the bursts are set to $\sigma = 3 L$, and the rates $\rho$ are generated from a lognormal distribution with $\mu = 0.8$~Gbps and $\sigma^2 = 0.05$ \cite{fnss}. The flow deadlines $\delta$ are chosen uniformly within the interval $[ \, \delta_{min} \,,\, \delta_{min}+(\delta_{max} - \delta_{min}) \pi \, ]$, where  $\delta_{min}$ and $\delta_{max}$ are a lower and upper bound on the WCD, respectively, and $\pi = 0.2$ is a fixed parameter. All the instances are available for download from \url{https://commalab.di.unipi.it/datasets/mmcf/#FNSS}.

\subsection{Experimental results}

We now report the results comparing our SB approach and {\tt Gurobi} on the different classes of instances. All the tables have the same structure. Each line refers to a specific network topology (column ``Topology'') in the class, reporting the number $n$ of nodes and $m$ of arcs, as well as the number $f$ of flows (i.e., size of the traffic matrix) that were generated for the corresponding topology via {\tt FNSS}. Each flow represents a different (SFSP) DCR instance. Among the $f$ generated flows, column $p_B$ gives the percentage solved to optimality (relative gap lower than {\tt1e-6}) by the SB algorithm. Column TR reports the ratio of the average running time of SB upon all the flows and that of {\tt Gurobi}, so that a value greater than one indicates SB being slower, while a value smaller than one indicates SB being faster. Column $it_B$ reports the average Benders iterations for the SB algorithm. The following columns report information about the quality of the solution obtained by SB:
\begin{itemize}
 \item $g^p$ is the average percentage {\em primal gap} between the SB algorithm and {\tt Gurobi}, {computed as $100\,|u_G-u_B|/u_G$};
 \item $g^d$ is the average percentage  {\em dual gap} between the SB algorithm and {\tt Gurobi} {computed as $100\,(u_G-\ell_B)/u_G$};
\end{itemize}
where $\ell_G$, $u_G$ are the best lower and upper bound found by {\tt Gurobi} (always identical up to {\tt 1e-6} relative), while $\ell_B$, $u_B$ are the best lower and upper bound found by our SB algorithm, respectively. The average is computed only over the instances not solved to optimality; if all are, ``--'' is shown.

\smallskip
\noindent
In Table \ref{tab:garr-results}, we report the results for the {GARR} class. SB solves to optimality all but six topologies, in each of which only among one to three flows get a nonzero gap; and in all cases, the obtained solution is always optimal (the primal gap is null), just the lower bound falls short from proving it with a residual gap of 5-7\%.
Except from five cases, SB is faster than Gurobi by up to a factor of three. Remarkably, in all cases but one, the Benders' approach does just one iteration, i.e., no branching is needed and the gap is closed at the first attempt.

\begingroup
\footnotesize
\setlength{\tabcolsep}{5pt}
\renewcommand{\arraystretch}{1.15}
\begin{longtable}{lrrrrrrrr}
\caption{Results for the GARR network.}\label{tab:garr-results}\\
\toprule
Topology & \multicolumn{1}{c}{$n$} & \multicolumn{1}{c}{$m$} & \multicolumn{1}{c}{$f$} & \multicolumn{1}{c}{$p_{B}$} & \multicolumn{1}{c}{TR} & \multicolumn{1}{c}{$it_{B}$} & \multicolumn{1}{c}{$g^p$} & \multicolumn{1}{c}{$g^d$} \\
\midrule
\endfirsthead
\toprule
Topology & \multicolumn{1}{c}{$n$} & \multicolumn{1}{c}{$m$} & \multicolumn{1}{c}{$f$} & \multicolumn{1}{c}{$p_{B}$} & \multicolumn{1}{c}{TR} & \multicolumn{1}{c}{$it_{B}$} & \multicolumn{1}{c}{$g^p$} & \multicolumn{1}{c}{$g^d$} \\
\midrule
\endhead
\midrule
\multicolumn{9}{r}{\textit{Continued on next page}} \\
\endfoot
\bottomrule
\endlastfoot
Garr199901 & 16 & 36 & 100 & 100 & 0.224 & 1 & \multicolumn{1}{c}{--} & \multicolumn{1}{c}{--} \\
Garr199904 & 23 & 50 & 100 & 100 & 1.661 & 1 & \multicolumn{1}{c}{--} & \multicolumn{1}{c}{--} \\
Garr199905 & 23 & 50 & 100 & 100 & 1.129 & 1 & \multicolumn{1}{c}{--} & \multicolumn{1}{c}{--} \\
Garr200109 & 22 & 48 & 100 & 100 & 1.662 & 1 & \multicolumn{1}{c}{--} & \multicolumn{1}{c}{--} \\
Garr200112 & 24 & 52 & 100 & 100 & 1.633 & 1 & \multicolumn{1}{c}{--} & \multicolumn{1}{c}{--} \\
Garr200212 & 27 & 56 & 100 & 100 & 1.269 & 1 & \multicolumn{1}{c}{--} & \multicolumn{1}{c}{--} \\
Garr200404 & 22 & 48 & 100 & 100 & 0.932 & 3 & \multicolumn{1}{c}{--} & \multicolumn{1}{c}{--} \\
Garr200902 & 54 & 136 & 100 & 100 & 0.561 & 1 & \multicolumn{1}{c}{--} & \multicolumn{1}{c}{--} \\
Garr200908 & 54 & 136 & 100 & 100 & 0.604 & 1 & \multicolumn{1}{c}{--} & \multicolumn{1}{c}{--} \\
Garr200909 & 55 & 138 & 100 & 100 & 0.378 & 1 & \multicolumn{1}{c}{--} & \multicolumn{1}{c}{--} \\
Garr200912 & 54 & 136 & 100 & 100 & 0.327 & 1 & \multicolumn{1}{c}{--} & \multicolumn{1}{c}{--} \\
Garr201001 & 54 & 136 & 100 & 100 & 0.327 & 1 & \multicolumn{1}{c}{--} & \multicolumn{1}{c}{--} \\
Garr201003 & 54 & 136 & 100 & 100 & 0.298 & 1 & \multicolumn{1}{c}{--} & \multicolumn{1}{c}{--} \\
Garr201004 & 54 & 136 & 100 & 100 & 0.291 & 1 & \multicolumn{1}{c}{--} & \multicolumn{1}{c}{--} \\
Garr201005 & 55 & 138 & 100 & 100 & 0.428 & 1 & \multicolumn{1}{c}{--} & \multicolumn{1}{c}{--} \\
Garr201007 & 55 & 138 & 100 & 100 & 0.290 & 1 & \multicolumn{1}{c}{--} & \multicolumn{1}{c}{--} \\
Garr201008 & 55 & 138 & 100 & 100 & 0.284 & 1 & \multicolumn{1}{c}{--} & \multicolumn{1}{c}{--} \\
Garr201010 & 56 & 140 & 100 & 100 & 0.273 & 1 & \multicolumn{1}{c}{--} & \multicolumn{1}{c}{--} \\
Garr201012 & 56 & 140 & 100 & 100 & 0.309 & 1 & \multicolumn{1}{c}{--} & \multicolumn{1}{c}{--} \\
Garr201101 & 56 & 140 & 100 & 100 & 0.478 & 1 & \multicolumn{1}{c}{--} & \multicolumn{1}{c}{--} \\
Garr201102 & 57 & 142 & 100 & 100 & 0.290 & 1 & \multicolumn{1}{c}{--} & \multicolumn{1}{c}{--} \\
Garr201103 & 58 & 144 & 100 & 100 & 0.615 & 1 & \multicolumn{1}{c}{--} & \multicolumn{1}{c}{--} \\
Garr201104 & 59 & 148 & 100 & 97 & 0.406 & 1 & 0.00 & 5.08 \\
Garr201105 & 59 & 148 & 100 & 99 & 0.485 & 1 & 0.00 & 7.36 \\
Garr201107 & 59 & 148 & 100 & 99 & 0.632 & 1 & 0.00 & 7.20 \\
Garr201108 & 59 & 148 & 100 & 99 & 0.438 & 1 & 0.00 & 7.04 \\
Garr201109 & 59 & 148 & 100 & 98 & 0.424 & 1 & 0.00 & 6.41 \\
Garr201110 & 59 & 148 & 100 & 98 & 0.615 & 1 & 0.00 & 6.53 \\
Garr201111 & 60 & 148 & 100 & 100 & 0.456 & 1 & \multicolumn{1}{c}{--} & \multicolumn{1}{c}{--} \\
Garr201112 & 61 & 150 & 100 & 100 & 0.463 & 1 & \multicolumn{1}{c}{--} & \multicolumn{1}{c}{--} \\
Garr201201 & 61 & 150 & 100 & 100 & 0.663 & 1 & \multicolumn{1}{c}{--} & \multicolumn{1}{c}{--} \\
\end{longtable}
\endgroup

\smallskip
\noindent
Table \ref{tab:sndlib-results} paints a similar picture for the {SNDlib} class: SB fails to solve eight topologies (in each case, by five flows or less) and in every case finds the optimal solution; however, the lower bound can be off by up to 22\%. The ratio of the running times is also more variable, with {\tt Gurobi} being faster up to a factor of four in half of the cases, and SB being faster by up to one order of magnitude in the other half. Benders' iterations are similarly very few.

\begingroup
\footnotesize
\setlength{\tabcolsep}{5pt}
\renewcommand{\arraystretch}{1.15}
\begin{longtable}{lrrrrrrrr}
\caption{Results for the SNDlib instances.}\label{tab:sndlib-results}\\
\toprule
Topology & \multicolumn{1}{c}{$n$} & \multicolumn{1}{c}{$m$} & \multicolumn{1}{c}{$f$} & \multicolumn{1}{c}{$p_{B}$} & \multicolumn{1}{c}{TR} & \multicolumn{1}{c}{$it_{B}$} & \multicolumn{1}{c}{$g^p$} & \multicolumn{1}{c}{$g^d$} \\
\midrule
\endfirsthead
\toprule
Topology & \multicolumn{1}{c}{$n$} & \multicolumn{1}{c}{$m$} & \multicolumn{1}{c}{$f$} & \multicolumn{1}{c}{$p_{B}$} & \multicolumn{1}{c}{TR} & \multicolumn{1}{c}{$it_{B}$} & \multicolumn{1}{c}{$g^p$} & \multicolumn{1}{c}{$g^d$} \\
\midrule
\endhead
\midrule
\multicolumn{9}{r}{\textit{Continued on next page}} \\
\endfoot
\bottomrule
\endlastfoot
abilene & 12 & 15 & 31 & 100 & 0.370 & 4 & \multicolumn{1}{c}{--} & \multicolumn{1}{c}{--} \\
atlanta & 15 & 22 & 45 & 100 & 0.650 & 1 & \multicolumn{1}{c}{--} & \multicolumn{1}{c}{--} \\
cost266 & 37 & 57 & 100 & 99 & 1.734 & 3 & 0.00 & 0.08 \\
dfn-bwin & 10 & 45 & 45 & 100 & 2.581 & 1 & \multicolumn{1}{c}{--} & \multicolumn{1}{c}{--} \\
dfn-gwin & 11 & 47 & 53 & 100 & 2.855 & 1 & \multicolumn{1}{c}{--} & \multicolumn{1}{c}{--} \\
di-yuan & 11 & 42 & 58 & 100 & 0.798 & 1 & \multicolumn{1}{c}{--} & \multicolumn{1}{c}{--} \\
france & 25 & 45 & 66 & 100 & 2.762 & 2 & \multicolumn{1}{c}{--} & \multicolumn{1}{c}{--} \\
geant & 22 & 36 & 63 & 100 & 1.040 & 1 & \multicolumn{1}{c}{--} & \multicolumn{1}{c}{--} \\
germany50 & 50 & 88 & 100 & 100 & 1.297 & 1 & \multicolumn{1}{c}{--} & \multicolumn{1}{c}{--} \\
giul39 & 39 & 172 & 100 & 100 & 0.073 & 1 & \multicolumn{1}{c}{--} & \multicolumn{1}{c}{--} \\
india35 & 35 & 80 & 100 & 98 & 2.334 & 1 & 0.00 & 9.97 \\
janos-us & 26 & 84 & 100 & 95 & 0.127 & 1 & 0.00 & 17.12 \\
janos-us-ca & 39 & 122 & 100 & 98 & 0.101 & 1 & 0.00 & 2.51 \\
newyork & 16 & 49 & 89 & 99 & 1.298 & 1 & 0.00 & 10.54 \\
nobel-eu & 28 & 41 & 100 & 100 & 1.707 & 1 & \multicolumn{1}{c}{--} & \multicolumn{1}{c}{--} \\
nobel-germany & 17 & 26 & 51 & 100 & 0.614 & 1 & \multicolumn{1}{c}{--} & \multicolumn{1}{c}{--} \\
nobel-us & 14 & 21 & 24 & 100 & 0.402 & 1 & \multicolumn{1}{c}{--} & \multicolumn{1}{c}{--} \\
norway & 27 & 51 & 100 & 100 & 0.390 & 2 & \multicolumn{1}{c}{--} & \multicolumn{1}{c}{--} \\
pdh & 11 & 34 & 54 & 98 & 0.270 & 5 & 0.00 & 0.18 \\
pioro40 & 40 & 89 & 100 & 100 & 4.354 & 1 & \multicolumn{1}{c}{--} & \multicolumn{1}{c}{--} \\
polska & 12 & 18 & 24 & 96 & 1.469 & 1 & 0.00 & 10.57 \\
sun & 27 & 102 & 100 & 97 & 0.095 & 5 & 0.00 & 22.17 \\
ta2 & 65 & 108 & 100 & 100 & 1.719 & 1 & \multicolumn{1}{c}{--} & \multicolumn{1}{c}{--} \\
\end{longtable}
\endgroup

\smallskip
\noindent
The TopoZoo instances in Table \ref{tab:topo-results} confirm the general trend, save that the majority of flows are not fully solved, in a few cases up to 28\%. Very rarely the upper bound is not exact, and always by a fraction of a percentage point; usually the average error in the lower bound (on the few flows that have one) is also a few percentage points, but cases with up to 40\% error also are present. The average number of Benders' iterations is more variable, reaching 100 in one case, but the majority of flows still terminates in one or very few ones.

\begingroup
\footnotesize
\setlength{\tabcolsep}{4pt}
\renewcommand{\arraystretch}{1.15}
\begin{longtable}{lrrrrrrrrr}
\caption{Results for the Topology Zoo instances.}\label{tab:topo-results}\\
\toprule
Topology & \multicolumn{1}{c}{$n$} & \multicolumn{1}{c}{$m$} & \multicolumn{1}{c}{$f$} & \multicolumn{1}{c}{$p_{B}$} & \multicolumn{1}{c}{TR} & \multicolumn{1}{c}{$it_{B}$} & \multicolumn{1}{c}{$g^p$} & \multicolumn{1}{c}{$g^d$} \\
\midrule
\endfirsthead
\toprule
Topology & \multicolumn{1}{c}{$n$} & \multicolumn{1}{c}{$m$} & \multicolumn{1}{c}{$f$} & \multicolumn{1}{c}{$p_{B}$} & \multicolumn{1}{c}{TR} & \multicolumn{1}{c}{$it_{B}$} & \multicolumn{1}{c}{$g^p$} & \multicolumn{1}{c}{$g^d$} \\
\midrule
\endhead
\midrule
\multicolumn{9}{r}{\textit{Continued on next page}} \\
\endfoot
\bottomrule
\endlastfoot
Agis & 25 & 60 & 100 & 99 & 0.063 & 5 & 0.00 & 13.34 \\
Arpanet196912 & 4 & 8 & 12 & 100 & 0.159 & 1 & \multicolumn{1}{c}{--} & \multicolumn{1}{c}{--} \\
Arpanet19723 & 25 & 56 & 100 & 97 & 0.083 & 9 & 0.00 & 25.26 \\
Arpanet19728 & 29 & 64 & 100 & 98 & 0.075 & 4 & 0.00 & 10.53 \\
AttMpls & 25 & 112 & 100 & 96 & 0.099 & 32 & 0.00 & 6.92 \\
Bellcanada & 48 & 128 & 100 & 96 & 0.429 & 1 & 0.00 & 9.08 \\
Belnet2003 & 23 & 78 & 100 & 100 & 0.113 & 1 & \multicolumn{1}{c}{--} & \multicolumn{1}{c}{--} \\
Belnet2004 & 23 & 78 & 100 & 100 & 0.113 & 1 & \multicolumn{1}{c}{--} & \multicolumn{1}{c}{--} \\
Belnet2005 & 23 & 82 & 100 & 100 & 0.118 & 1 & \multicolumn{1}{c}{--} & \multicolumn{1}{c}{--} \\
Belnet2006 & 23 & 82 & 100 & 100 & 0.095 & 1 & \multicolumn{1}{c}{--} & \multicolumn{1}{c}{--} \\
Bics & 33 & 96 & 100 & 95 & 0.135 & 1 & 0.00 & 15.83 \\
Biznet & 29 & 66 & 100 & 96 & 0.112 & 1 & 0.00 & 7.36 \\
BtEurope & 24 & 74 & 100 & 100 & 0.113 & 4 & \multicolumn{1}{c}{--} & \multicolumn{1}{c}{--} \\
BtNorthAmerica & 36 & 152 & 100 & 97 & 0.107 & 18 & 0.00 & 0.64 \\
Cernet & 41 & 116 & 100 & 100 & 0.165 & 1 & \multicolumn{1}{c}{--} & \multicolumn{1}{c}{--} \\
Cesnet200304 & 29 & 66 & 100 & 98 & 0.215 & 1 & 0.00 & 24.31 \\
Chinanet & 42 & 132 & 100 & 100 & 0.137 & 1 & \multicolumn{1}{c}{--} & \multicolumn{1}{c}{--} \\
Cogentco & 197 & 486 & 100 & 100 & 0.443 & 1 & \multicolumn{1}{c}{--} & \multicolumn{1}{c}{--} \\
Colt & 153 & 354 & 100 & 99 & 0.703 & 1 & 0.00 & 0.92 \\
Columbus & 70 & 170 & 100 & 87 & 0.272 & 84 & 0.00 & 8.09 \\
Cwix & 36 & 82 & 100 & 98 & 0.102 & 1 & 0.00 & 20.70 \\
Deltacom & 113 & 322 & 100 & 95 & 0.481 & 9 & 0.00 & 13.25 \\
DeutscheTelekom & 39 & 124 & 100 & 100 & 0.118 & 1 & \multicolumn{1}{c}{--} & \multicolumn{1}{c}{--} \\
Dfn & 58 & 174 & 100 & 99 & 0.195 & 1 & 0.00 & 3.49 \\
DialtelecomCz & 193 & 302 & 100 & 89 & 0.274 & 58 & 0.00 & 12.66 \\
Digex & 31 & 70 & 100 & 98 & 0.091 & 23 & 0.20 & 7.29 \\
EliBackbone & 20 & 60 & 100 & 92 & 0.123 & 1 & 0.00 & 14.54 \\
Esnet & 68 & 158 & 100 & 96 & 0.413 & 1 & 0.00 & 23.12 \\
Evolink & 37 & 90 & 100 & 100 & 0.222 & 5 & \multicolumn{1}{c}{--} & \multicolumn{1}{c}{--} \\
Geant2001 & 27 & 76 & 100 & 100 & 0.134 & 1 & \multicolumn{1}{c}{--} & \multicolumn{1}{c}{--} \\
Geant2009 & 34 & 104 & 100 & 98 & 0.184 & 1 & 0.00 & 18.21 \\
Geant2012 & 40 & 122 & 100 & 100 & 0.212 & 1 & \multicolumn{1}{c}{--} & \multicolumn{1}{c}{--} \\
Globenet & 67 & 190 & 100 & 99 & 0.575 & 1 & 0.00 & 0.76 \\
Goodnet & 17 & 62 & 100 & 100 & 0.055 & 5 & \multicolumn{1}{c}{--} & \multicolumn{1}{c}{--} \\
Grnet & 37 & 84 & 100 & 98 & 0.135 & 1 & 0.00 & 9.67 \\
GtsCe & 149 & 386 & 100 & 100 & 0.340 & 1 & \multicolumn{1}{c}{--} & \multicolumn{1}{c}{--} \\
GtsPoland & 33 & 74 & 100 & 93 & 0.093 & 5 & 0.00 & 29.15 \\
Heanet & 7 & 22 & 42 & 95 & 0.068 & 1 & 0.00 & 39.88 \\
HiberniaGlobal & 55 & 162 & 100 & 99 & 0.486 & 1 & 0.00 & 19.60 \\
HiberniaNireland & 18 & 42 & 100 & 99 & 0.082 & 1 & 0.00 & 38.39 \\
HiberniaUk & 15 & 30 & 100 & 95 & 0.102 & 11 & 0.18 & 17.75 \\
HiberniaUs & 22 & 58 & 100 & 97 & 0.097 & 1 & 0.00 & 19.09 \\
Highwinds & 18 & 62 & 100 & 98 & 0.159 & 1 & 0.26 & 1.77 \\
Ibm & 18 & 48 & 100 & 100 & 0.117 & 6 & \multicolumn{1}{c}{--} & \multicolumn{1}{c}{--} \\
Iij & 37 & 130 & 100 & 97 & 0.160 & 2 & 0.00 & 18.97 \\
Integra & 27 & 72 & 100 & 98 & 0.123 & 12 & 0.00 & 0.27 \\
Intellifiber & 73 & 190 & 100 & 100 & 0.291 & 1 & \multicolumn{1}{c}{--} & \multicolumn{1}{c}{--} \\
Internetmci & 19 & 66 & 100 & 94 & 0.324 & 67 & 0.00 & 18.02 \\
Interoute & 110 & 292 & 100 & 72 & 0.327 & 81 & 0.00 & 11.21 \\
Intranetwork & 39 & 102 & 100 & 100 & 0.919 & 1 & \multicolumn{1}{c}{--} & \multicolumn{1}{c}{--} \\
Ion & 125 & 292 & 100 & 95 & 0.167 & 12 & 0.00 & 18.27 \\
IowaStatewideFiberMap & 33 & 82 & 100 & 96 & 0.108 & 2 & 0.00 & 23.60 \\
Janetlense & 20 & 68 & 100 & 98 & 0.114 & 1 & 0.00 & 12.94 \\
LambdaNet & 42 & 92 & 100 & 98 & 0.176 & 3 & 0.50 & 6.01 \\
Missouri & 67 & 166 & 100 & 94 & 0.448 & 1 & 0.29 & 7.02 \\
Netrail & 7 & 20 & 42 & 98 & 0.045 & 10 & 0.00 & 32.68 \\
NetworkUsa & 35 & 78 & 100 & 100 & 0.136 & 1 & \multicolumn{1}{c}{--} & \multicolumn{1}{c}{--} \\
Nextgen & 17 & 38 & 100 & 99 & 0.176 & 1 & 0.00 & 53.54 \\
Ntelos & 48 & 116 & 100 & 92 & 0.230 & 2 & 0.00 & 15.12 \\
Ntt & 47 & 126 & 100 & 100 & 0.088 & 1 & \multicolumn{1}{c}{--} & \multicolumn{1}{c}{--} \\
Oteglobe & 93 & 206 & 100 & 100 & 0.318 & 1 & \multicolumn{1}{c}{--} & \multicolumn{1}{c}{--} \\
Palmetto & 45 & 128 & 100 & 89 & 0.240 & 3 & 0.00 & 21.43 \\
Pern & 127 & 258 & 100 & 99 & 1.779 & 1 & 0.00 & 1.46 \\
PionierL1 & 36 & 82 & 100 & 97 & 0.443 & 7 & 0.00 & 17.21 \\
Psinet & 24 & 50 & 100 & 100 & 0.173 & 1 & \multicolumn{1}{c}{--} & \multicolumn{1}{c}{--} \\
Quest & 20 & 62 & 100 & 100 & 0.101 & 1 & \multicolumn{1}{c}{--} & \multicolumn{1}{c}{--} \\
RedBestel & 84 & 186 & 100 & 91 & 0.359 & 52 & 0.00 & 4.47 \\
Rediris & 19 & 62 & 100 & 97 & 0.274 & 34 & 0.00 & 0.05 \\
Renater2010 & 43 & 112 & 100 & 99 & 0.119 & 8 & 0.00 & 21.82 \\
Sanet & 43 & 90 & 100 & 93 & 0.118 & 50 & 0.00 & 15.41 \\
Shentel & 28 & 70 & 100 & 98 & 0.500 & 5 & 0.00 & 15.03 \\
Spiralight & 15 & 32 & 100 & 99 & 0.127 & 3 & \multicolumn{1}{c}{--} & \multicolumn{1}{c}{--} \\
Sunet & 26 & 64 & 100 & 94 & 0.061 & 13 & 0.18 & 17.18 \\
Surfnet & 50 & 136 & 100 & 98 & 0.205 & 1 & 0.00 & 26.00 \\
Switch & 74 & 184 & 100 & 97 & 0.284 & 16 & 0.00 & 13.72 \\
Syringa & 74 & 148 & 100 & 93 & 0.268 & 104 & 0.00 & 0.18 \\
TataNld & 145 & 372 & 100 & 95 & 0.139 & 1 & 0.00 & 16.50 \\
Tw & 76 & 230 & 100 & 82 & 0.731 & 1 & 0.00 & 16.50 \\
Uninett2010 & 74 & 202 & 100 & 100 & 0.214 & 1 & \multicolumn{1}{c}{--} & \multicolumn{1}{c}{--} \\
Uninett2011 & 69 & 192 & 100 & 99 & 0.209 & 1 & 0.00 & 1.55 \\
UsCarrier & 158 & 378 & 100 & 99 & 0.654 & 1 & 0.00 & 4.57 \\
UsSignal & 63 & 156 & 100 & 97 & 0.219 & 1 & 0.23 & 4.73 \\
Uunet & 49 & 168 & 100 & 97 & 0.135 & 58 & 0.00 & 15.43 \\
VtlWavenet2008 & 88 & 184 & 100 & 95 & 0.184 & 46 & 0.00 & 7.16 \\
VtlWavenet2011 & 92 & 192 & 100 & 93 & 0.205 & 41 & 0.00 & 11.65 \\
York & 23 & 48 & 100 & 95 & 0.172 & 14 & 0.00 & 3.30 \\
\end{longtable}
\endgroup

\smallskip
\noindent
Finally, Table \ref{tab:waxman-results} shows the results of the Waxman instances, the larger and therefore more challenging topologies. While only three classes are fully solved, the upper bound is always exact, although the dual gap can still be very significant, up to around 25\%. The number of Benders' iteration is, in most cases, not negligible, reaching the several 100s; however, the running time now very clearly favours SB, which is always faster by at least a factor of 30 and invariably reaching around a factor of 300 on the largest instances.

\begingroup
\footnotesize
\setlength{\tabcolsep}{5pt}
\renewcommand{\arraystretch}{1.15}
\begin{longtable}{lrrrrrrrr}
\caption{Results for the Waxman random instances.}\label{tab:waxman-results}\\
\toprule
Topology & \multicolumn{1}{c}{$n$} & \multicolumn{1}{c}{$m$} & \multicolumn{1}{c}{$f$} & \multicolumn{1}{c}{$p_{B}$} & \multicolumn{1}{c}{TR} & \multicolumn{1}{c}{$it_{B}$} & \multicolumn{1}{c}{$g^p$} & \multicolumn{1}{c}{$g^d$} \\
\midrule
\endfirsthead
\toprule
Topology & \multicolumn{1}{c}{$n$} & \multicolumn{1}{c}{$m$} & \multicolumn{1}{c}{$f$} & \multicolumn{1}{c}{$p_{B}$} & \multicolumn{1}{c}{TR} & \multicolumn{1}{c}{$it_{B}$} & \multicolumn{1}{c}{$g^p$} & \multicolumn{1}{c}{$g^d$} \\
\midrule
\endhead
\midrule
\multicolumn{9}{r}{\textit{Continued on next page}} \\
\endfoot
\bottomrule
\endlastfoot
w1\_100\_04 & 100 & 414 & 100 & 95 & 0.029 & 1 & 0.00 & 20.14 \\
w1\_200\_03 & 200 & 1536 & 100 & 100 & 0.011 & 1 & \multicolumn{1}{c}{--} & \multicolumn{1}{c}{--} \\
w1\_300\_04 & 300 & 3630 & 100 & 92 & 0.008 & 174 & 0.00 & 23.28 \\
w1\_300\_08 & 300 & 3564 & 100 & 90 & 0.008 & 308 & 0.00 & 22.55 \\
w1\_400\_04 & 400 & 6660 & 100 & 86 & 0.004 & 265 & 0.00 & 24.96 \\
w1\_400\_08 & 400 & 6384 & 100 & 100 & 0.011 & 1 & \multicolumn{1}{c}{--} & \multicolumn{1}{c}{--} \\
w1\_500\_04 & 500 & 9738 & 100 & 93 & 0.003 & 328 & 0.00 & 1.39 \\
w1\_500\_08 & 500 & 10148 & 100 & 99 & 0.005 & 1 & 0.00 & 10.07 \\
w1\_600\_04 & 600 & 14156 & 100 & 92 & 0.003 & 148 & 0.00 & 16.42 \\
w1\_600\_08 & 600 & 14116 & 100 & 97 & 0.004 & 1 & 0.00 & 26.60 \\
w1\_700\_04 & 700 & 19476 & 100 & 100 & 0.004 & 1 & \multicolumn{1}{c}{--} & \multicolumn{1}{c}{--} \\
w1\_700\_08 & 700 & 19376 & 100 & 88 & 0.003 & 374 & 0.00 & 21.03 \\
w1\_800\_04 & 800 & 25444 & 100 & 89 & 0.003 & 416 & 0.00 & 27.30 \\
w1\_800\_08 & 800 & 25628 & 100 & 95 & 0.003 & 68 & 0.00 & 23.44 \\
\end{longtable}
\endgroup

\smallskip
\noindent
The results therefore show that, especially on large-scale networks, the SB algorithm always provides extremely accurate feasible solutions, very often (but not always) being able to actually prove their optimality to a very high degree of confidence, in a running time that is at least one order of magnitude, and up to over two orders of magnitude, smaller than that of a state-of-the-art MINLP solver when solving the best known formulation.

\section{Conclusions}\label{sec:thend}

Delay constrained routing (DCR) problems arise frequently in telecomunnication networks. For the single-flow single-path case analysed in this paper (SFSP) DCR, we have shown that it is possible to devise an effective exact solver-free algorithm based on Benders decomposition and spatial branching; the approach is freely available to all interested parties in the context of the open-source {\tt SMS++} modelling framework. There are several interesting topics for future research. First, we would like to exploit the obtained results to tackle the multiflow case in which several flows are routed simultaneously subject to mutual capacity constraints: this is useful to periodically globally optimize a network after a sequence of quick, single-flow greedy decisions. Second, it would be interesting to devise a branch-bound component in order to try to solve to optimality the (relatively few) instances for which the final inherent gap of the SB algorithm is not zero. A possible approach would be to branch on the fractional variables produced at optimality by the Lagrangian dual, but since the approach hinges on shortest path computations, the standard branching rules cannot be applied (finding the shortest path passing from a prescribed set of arcs is $\mathcal{NP}$-hard) and ad-hoc ones should be devised. Also, since there would be two nested enumerative approaches, it is not at all trivial which of the two enumerations should take precedence on the other. Finally, one could consider a robust version of the problem in which the flows arrive in a stochastic manner.

\bibliography{dcrbib}
\bibliographystyle{plain}

\end{document}